\RequirePackage[l2tabu, orthodox]{nag}
\documentclass[a4paper,12pt]{amsart}
\usepackage{microtype}
\usepackage[charter]{mathdesign}

\makeatletter
\newcommand{\leqnomode}{\tagsleft@true\let\veqno\@@leqno}
\makeatother
\usepackage[T1]{fontenc}
\usepackage{tikz, tikz-cd, stmaryrd, amsmath, amsthm, amssymb,
	hyperref, bbm, mathtools, mathrsfs}	
\usepackage[all]{xy}

\usepackage[shortlabels]{enumitem}
\usetikzlibrary{arrows}
\usetikzlibrary{positioning}
\usepackage[utf8x]{inputenc}
\newcommand{\bb}{\textbf}

\newcommand{\ol}{\overline}
\newcommand{\mc}{\mathcal}

\newcommand{\ZZ}{\mathbb{Z}}

\newcommand{\PP}{\mathbb{P}}
\newcommand{\QQ}{\mathbb{Q}}

\newcommand{\FF}{\mathbb{F}}

\newcommand{\JJ}{\mathbb{J}}
\newcommand{\DD}{\mathbb{D}}

\newcommand{\MM}{\mathbb{M}}

\renewcommand{\SS}{\mathbb{S}}
\newcommand{\Mod}{\mathbf{Mod}}
\newcommand{\YD}{\mathbf{YD}}
\newcommand{\msm}{\mathsf{m}}
\newcommand{\msM}{\mathsf{M}}

\newcommand{\et}{{\rm\acute{e}t}}

\DeclareMathOperator{\Ind}{Ind}
\DeclareMathOperator{\Tor}{Tor}
\DeclareMathOperator{\Fix}{Fix}
\DeclareMathOperator{\pr}{pr}
\DeclareMathOperator{\tr}{tr}

\DeclareMathOperator{\coker}{coker}
\DeclareMathOperator{\im}{im}
\DeclareMathOperator{\id}{id}

\DeclareMathOperator{\Span}{Span}
\DeclareMathOperator{\Res}{Res}
\DeclareMathOperator{\res}{res}
\DeclareMathOperator{\cores}{cores}
\DeclareMathOperator{\GCD}{GCD}
\DeclareMathOperator{\ord}{ord}
\DeclareMathOperator{\rank}{rank}

\DeclareMathOperator{\len}{len}

\DeclareMathOperator{\cris}{cris}

\theoremstyle{plain}
\newtheorem{Theorem}{Theorem}[section]

\newtheorem{Lemma}[Theorem]{Lemma}
\newtheorem{Corollary}[Theorem]{Corollary}

\newtheorem{Proposition}[Theorem]{Proposition}

\theoremstyle{definition}
\newtheorem{Definition}[Theorem]{Definition}

\numberwithin{equation}{section}
\begin{document}
	
\title[The crystalline cohomology of covers]{The crystalline cohomology of covers\\ with a cyclic $p$-Sylow subgroup}
\author[J. Garnek]{J\k{e}drzej Garnek}

\date{}  

\begin{abstract}
	Let $X$ be a smooth projective curve over a field $k$ with an action of a finite group $G$. A well-known result of Chevalley and Weil describes the $k[G]$-module structure of cohomologies of~$X$ in the case when the characteristic of $k$ does not divide $\# G$.
	In case when $G$ has a cyclic $p$-Sylow subgroup, it is known that the $G$-structures of the module of holomorphic differentials and of de Rham cohomology of~$X$ are completely determined by the ramification data of the cover $X \to X/G$.
	In this article we extend this result to the crystalline cohomology of~$X$. Also, we
	provide an explicit description of the structure of the crystalline cohomology when $G = \ZZ/p^n$. The main used tool is the theory of Yakovlev diagrams -- algebraic objects that classify the $G$-modules over Witt vectors. 
\end{abstract}

\maketitle
\bibliographystyle{plain}

\section{Introduction} \label{sec:intro}
Let $k$ be a perfect field of characteristic $p > 0$ and suppose that $X$ is a smooth projective curve over~$k$ with an action of a finite group $G$. 
There are many results concerning the equivariant structure of cohomologies of~$X$ for particular groups
(see e.g.~\cite{Valentini_Madan_Automorphisms} for the case of cyclic groups, \cite{WardMarques_HoloDiffs} for abelian groups, \cite{Bleher_Chinburg_Kontogeorgis_Galois_structure} for groups with a cyclic Sylow subgroup, or \cite{Bleher_Camacho_Holomorphic_differentials} for the Klein group) or curves (cf. \cite{Lusztig_Coxeter_orbits}, \cite{Dummigan_99}, \cite{Gross_Rigid_local_systems_Gm}, \cite{laurent_kock_drinfeld}). Also, one may expect that (at least in the case of $p$-groups) determining cohomologies comes down to Harbater--Katz--Gabber covers (cf. \cite{Garnek_p_gp_covers}, \cite{Garnek_p_gp_covers_ii}). However, there is no hope of obtaining a general formula. Indeed, if $G$ is a group with a non-cyclic $p$-Sylow subgroup, the set of indecomposable $k[G]$-modules is infinite. If, moreover, $p > 2$ then the indecomposable $k[G]$-modules are considered impossible to classify (cf. \cite{Prest}).
This brings attention to groups with a cyclic $p$-Sylow subgroup. For those, the set of 
equivalence classes of indecomposable modules is finite (cf. \cite{Higman}, \cite{Borevic_Faddeev}, \cite{Heller_Reiner_Reps_in_integers_I}). While their representation theory still seems a bit too complicated to derive a general formula for the cohomologies, 
the articles~\cite{Bleher_Chinburg_Kontogeorgis_Galois_structure} and \cite{Garnek_Kontogeorgis_cyclic_de_Rham} show that 
the $k[G]$-module structure of $H^0(X, \Omega_X)$ and $H^1_{dR}(X)$ are determined by the genus of $X$ and \emph{ramification data}, that is:
\begin{itemize}
	\item the stabilizers of the action on $X(\ol k)$,
	
	\item the lower (or equivalently, upper) ramification jumps of the action corresponding to the
	stabilizers,
	
	\item the fundamental characters of the action.
\end{itemize}
The de Rham cohomology serves as a starting point for several other cohomology theories, including crystalline, prismatic, syntomic, and rigid cohomology. In this article we 
focus on the first of them. Crystalline cohomology has found many applications, for instance in the study of supersingularity, point counting, 
$p$-adic Hodge theory, and other aspects of the geometry of varieties in characteristic~$p$.\\

Let $W$ be the ring of Witt vectors of infinite length over $k$
and let~$K$ be its field of fractions. Let also $g_X$ denote the genus of the curve~$X$. Recall that $H^1_{\cris}(X) := H^1_{\cris}(X/W)$, the crystalline cohomology of~$X$ over~$W$, is a free $W$-module of rank $2 g_X$. Its reduction $H^1_{\cris}(X) \otimes_W k$ can be naturally identified with $H^1_{dR}(X)$.
Assume that a finite group $G$ acts on the curve~$X$. By functoriality, $H^1_{\cris}(X)$ becomes then a $W[G]$-module.
For the rational crystalline cohomology $H^1_{\cris}(X/K) := H^1_{\cris}(X) \otimes_W K$ the $G$-structure might be fully described in terms of characters, using the Lefschetz trace formula (see \cite[Corollary~V.2.8]{Milne1980} for a proof that works for any Weil cohomology theory). It seems, however that the equivariant structure of 
$H^1_{\cris}(X)$ is much more complicated, for instance because it encompasses the equivariant structure of $H^1_{dR}(X)$. We do not know of any examples of the equivariant structure of $H^1_{\cris}(X)$ computed explicitly for a given curve.
Even though there are several computationally feasible approaches to crystalline cohomology, it seems hard to find an explicit basis of $H^1_{\cris}(X)$. Standard methods, such as comparison with Monsky--Washnitzer cohomology
(see e.g.~\cite{Kedlaya2001}) or with
rigid cohomology (see~\cite{Booher_crystalline}), typically give access only to
$H^1_{\cris}(X/K)$. Integral methods based on de Rham--Witt
cohomology lead to difficult explicit computations with de Rham--Witt forms already for $H^1_{\cris}(X/W_2)$.\\

We now present the main result of the article, namely the analogue of \cite[Main Theorem]{Garnek_Kontogeorgis_cyclic_de_Rham} for the crystalline cohomology.
Let $W_n$ be the ring of Witt vectors of length $n$, i.e. $W_n := W/p^n W$.
\begin{Theorem} \label{thm:cyclic_sylow}
	Suppose that $G$ is a group with a cyclic $p$-Sylow subgroup~$F$.
	Let $X$ be a smooth projective curve with an action of~$G$ over a perfect field~$k$ of characteristic~$p$.
	The $W[G]$-module structure of $H^1_{\cris}(X)$ is uniquely determined by the ramification data of the cover $X \to X/G$ and the genus of $X$.
\end{Theorem}
Interestingly, in the proof of Theorem~\ref{thm:cyclic_sylow} we use only formal properties of the crystalline cohomology. The key input is
the theory of Yakovlev diagrams, which allows to describe the $W[G]$-modules for a group $G$ with a cyclic $p$-Sylow subgroup. The crucial case in the proof of Theorem~\ref{thm:cyclic_sylow} is $G = F \cong \ZZ/p^n$.
For this group we give an explicit description of the equivariant structure of the crystalline cohomology. Let $F = \langle \sigma \rangle$ and write $F_t$ for the unique subgroup of $F$ of order $p^t$. We define the following
$W[F]$-modules for every $i = 0, \ldots, n$:
\begin{equation*}
	\JJ_i := \ker(W[F] \to W[F/F_i]),
\end{equation*}
where $W[F] \to W[F/F_i]$ is the canonical projection. 
Similarly, we put for every $i = 0, \ldots, n$:
\begin{equation*}
	\SS_i := \ker(W[F] \to W[F/F_i]/W),
\end{equation*}
where $W$ is the image of the map $N_{F/F_i} := \sum_{g \in F/F_i} g : W \to W[F/F_i]$.
It turns out that $\JJ_i$ and $\SS_i$ are free $W$-modules of ranks $p^n - p^{n-i}$ and $p^n - p^{n-i} + 1$ respectively, and are indecomposable as $W[F]$-modules. We give a more explicit description of those modules in Section~\ref{sec:J_and_I}.

Assume that $\pi : X \to Y$ is an $F$-cover of smooth projective curves.
For any $P \in X(\ol k)$ we write $e_P$ and $u_P^{(t)}$,
for the ramification index of~$\pi$ at~$P$ and the $t$-th upper ramification jump of~$\pi$
at~$P$, respectively. Write $e_P = p^{m_P}$ and $m := \max \{ m_P : P \in X(\ol k) \}$. For any $Q \in Y(\ol k)$ we denote also by abuse of notation $e_Q := e_P$, $u_Q^{(t)} := u_P^{(t)}$, $m_Q := m_P$, etc. for arbitrary $P \in \pi^{-1}(Q)$. We put $u_Q^{(0)}:=1$.
\begin{Theorem} \label{thm:crys_cyclic}
Keep the above notation with $F \cong \ZZ/p^n$.
Assume that $\pi : X \to Y$ is an $F$-cover of smooth projective curves over $k$.
Let $B \subset Y(\ol k)$ be the branch set of $\pi$ and $g_Y$ be the genus of $Y$. Pick any point $Q_0 \in Y(\ol k)$ such that $m = m_{Q_0}$. Then as $W[F]$-modules:
\begin{equation*}
	H^1_{\cris}(X) \cong W[F]^{2 (g_Y - 1)} \oplus \SS_m^2 \oplus \bigoplus_{\substack{Q \in B\\ Q \neq Q_0}} \JJ_{m_Q}^2
	\oplus \bigoplus_{Q \in B} \bigoplus_{t = 0}^{m_Q  - 1} \JJ_{m_Q - t}^{u_Q^{(t+1)} - u_Q^{(t)}}.
\end{equation*}
\end{Theorem}
\noindent Note that this formula is well-defined for $g_Y = 0$, even though the first exponent is negative. Indeed, since in this case $m = n$ (as $\PP^1$ doesn't have any \'{e}tale covers), we have $\SS_m = W[F]$ and the first two summands cancel out.
The proof of Theorem~\ref{thm:crys_cyclic} is inductive. One 
uses the inductive hypothesis for the curve $X' := X/F_1$ with the action of $F/F_1$. 
Other important ingredients are the trace map $\tr_{X/Y} : H^1_{\cris}(X) \to H^1_{\cris}(Y)$ and its surjectivity, and the fact that we know the equivariant structure of $H^1_{\cris}(X) \otimes k \cong H^1_{dR}(X)$.
\subsection*{Outline of the paper}
Section~\ref{sec:notation} contains some notation and preliminaries concerning algebraic
curves and modules over~$W$. In Section~\ref{sec:yd} we review the group cohomology
for $W[F]$-modules and the theory of Yakovlev diagrams. 
The next section discusses the proof of the main results of Yakovlev, and explains
how to extend them to the case of Witt vectors. In Section~\ref{sec:J_and_I} we discuss various properties of the $W[F]$-modules $\JJ_i$ and~$\SS_i$, including their Yakovlev diagrams and their reductions.
Sections~\ref{sec:pf_of_cyclic} and~\ref{sec:pf_of_cyclic2} are devoted to the proofs of Theorem~\ref{thm:crys_cyclic} and Theorem~\ref{thm:cyclic_sylow}.

\subsection*{Acknowledgements} The author wishes to express his thanks to Aristides Kontogeorgis and Alexios Terezakis for many stimulating conversations.
The author also expresses his sincere gratitude to his wife and to Wojciech Gajda for their constant support. SageMath was used for computations mostly related to linear algebra. GPT-5.5 was used to detect several errors in earlier versions of the manuscript, as well as to help simplify the proof of Proposition~\ref{prop:lifting_F_automorphisms}.
The author was supported by the research grant SONATA 20 ``Symmetries of curves in positive characteristic'' UMO-2024/55/D/ST1/01377 awarded by the National Science Centre, Poland.

\section{Notation and preliminaries} \label{sec:notation}
Let $G$ be a finite group.
Assume that $\pi : X \to Y$ is a $G$-cover of smooth projective curves over a~perfect field $k$
of characteristic $p$.
Throughout the paper, we will use the following notation for any $P \in X(\ol k)$:
\begin{itemize}
	\item $e_{X/Y, P}$ is the ramification index at $P$,
	
	\item $m_{X/Y, P} := \ord_p(e_{X/Y, P})$ is the $p$-adic valuation of the ramification index at~$P$,
	
	\item $m_{X/Y} := \max \{ m_{X/Y, P} : P \in X(\ol k) \}$,
	
	\item $u_{X/Y, P}^{(t)}$ is the $t$-th upper ramification jump
	at $P$ for $1 \le t \le m_{X/Y, P}$,
	
	\item $u^{(0)}_{X/Y, P} := 1$ for any ramified point $P \in X(\ol k)$
	(note that this is not a standard convention),
	
	\item $u_{X/Y, P} := u_{X/Y, P}^{(m_{X/Y, P})}$ is the last ramification jump.
\end{itemize}
By the Hasse--Arf theorem (cf. \cite[p. 76]{Serre1979}),
if the inertia subgroup at~$P$ is abelian, the numbers $u_{X/Y, P}^{(t)}$ are integers.
For any $Q \in Y(\ol k)$ we denote also by abuse of notation $e_{X/Y, Q} := e_{X/Y, P}$,
$u_{X/Y, Q}^{(t)} := u_{X/Y, P}^{(t)}$, etc. for arbitrary $P \in \pi^{-1}(Q)$.
Let
\[
B_{X/Y} := \{ Q \in Y(\ol k) : e_{X/Y, Q} > 1 \}
\]
be the ramification locus of $\pi$.\\ 

We write $W := W(k)$ for the ring of Witt vectors of infinite length over $k$
and put $W_i := W/p^i W$. Recall that any $W$-module $M$ of finite length is isomorphic to a module of the form $\bigoplus_{i = 1}^n W_{a_i}$ for some uniquely determined integers $1 \le a_1 \le a_2 \le \cdots \le a_n$ (this follows e.g. from \cite[Theorem~12.1.5]{DummitFoote2004}). For any $W$-module $M$
we write $M[p^i] := \{m \in M : p^i \cdot m = 0 \}$. The following simple lemma
will be useful in the sequel.

\begin{Lemma} \label{lem:pM_and_M/p_determines_M}
	Let $M_1$, $M_2$ be $W$-modules of finite length. If $p \cdot M_1 \cong p \cdot M_2$
	and $\dim_k (M_1 \otimes_W k) = \dim_k (M_2 \otimes_W k)$ then $M_1 \cong M_2$.
\end{Lemma}
\begin{proof}
	Assume that $M_1 \cong W_1^a \oplus \bigoplus_{i = 1}^n W_{a_i}$ and $M_2 \cong W_1^b \oplus \bigoplus_{i = 1}^m W_{b_i}$, where $2 \le a_1 \le a_2 \le \cdots \le a_n$ and $2 \le b_1 \le b_2 \le \cdots \le b_m$. Then:
	\[
		\bigoplus_{i = 1}^n W_{a_i - 1} \cong p \cdot M_1 \cong p \cdot M_2 \cong \bigoplus_{i = 1}^m W_{b_i - 1},
	\]
	which yields $n = m$ and $(a_1, \ldots, a_n) = (b_1, \ldots, b_m)$. On the other hand,
	$a + n = \dim_k (M_1 \otimes_W k) = \dim_k (M_2 \otimes_W k) = b + m$. Hence $a = b$ and $M_1 \cong M_2$.
\end{proof}
\section{Yakovlev diagrams} \label{sec:yd}
Let $F$ be a cyclic group of order $p^n$ generated by $\sigma$. Write $F_t$ for its unique subgroup of order $p^t$ for $t = 0, \ldots, n$ and
let $\sigma_t := \sigma^{p^{n-t}}$ be its generator. We often abbreviate $F' := F/F_1$.
Write also $\Mod_{W[F]}$ (resp. $\Mod_{W[F]}^{\mathrm{ff}}$) for the category of $W[F]$-modules that are finitely generated over~$W$
(resp. $W[F]$-modules, that are finite free as $W$-modules).
For any group $G$, we write $N_G := \sum_{g \in G} g \in W[G]$. Let also $N_{F_{t+1}/F_t} := \sum_{i = 0}^{p-1} \sigma_{t+1}^i \in W[F]$. Note that
$N_{F_{t+1}} = N_{F_{t+1}/F_t} \cdot N_{F_t}$. \\

Let $G$ be a finite group with a normal subgroup $H$ and let $M$ be any $W$-module.
Identifying $\Ind^G M \cong \bigoplus_{g \in G} g \cdot M$, the quotient map
$q_{G, H} : G \to G/H$ induces a map
\[
	\Ind^G M \to \Ind^{G/H} M, \qquad \sum_{g \in G} g \cdot m_g \mapsto \sum_{g \in G} q_{G, H}(g) \cdot m_g,
\]
which we also denote by $q_{G, H}$. Similarly, there exists a map $\rho_{G, H} : \Ind^{G/H} M \to \Ind^G M$ defined
by $\ol g \cdot x \mapsto \sum_{g \in q_{G, H}^{-1}(\ol g)} g \cdot x$.
\subsection{Group cohomology}
We now recall some basic facts concerning the group cohomology of a $W[F]$-module $M$; see \cite{Serre1979} and \cite[Chapter~6]{Weibel} for a reference. The group homology and cohomology of $M$ are periodic, i.e. for any $i > 0$:
\begin{equation} \label{eqn:periodicity}
	H^i(F, M) \cong H_{i+1}(F, M) \cong
	\begin{cases}
		H^1(F, M), & 2 \nmid i,\\
		H^2(F, M), & 2 \mid i,
	\end{cases}
\end{equation}
(see \cite[VIII \S 4, p. 133]{Serre1979}). Moreover, we have:
\begin{equation}
	H^1(F_t, M) = \frac{\ker(N_{F_t} : M \to M)}{(\sigma_t - 1) M}. \label{eqn:H1_definition}
\end{equation}
Under the identification~\eqref{eqn:H1_definition}, the restriction and corestriction maps
are given by the formulas:
\begin{align*}
	\res \colon H^1(F_{t+1}, M) &\longrightarrow H^1(F_t, M),
	&&[x] \longmapsto [N_{F_{t+1}/F_t} \cdot x], \\
	\cores \colon H^1(F_t, M) &\longrightarrow H^1(F_{t+1}, M),
	&&[x] \longmapsto [x],
\end{align*}
where $[m]$ denotes the image of $m \in M$ in $H^1(F_t, M)$. Note that $\cores \circ \res = p$, while $\res \circ \cores = N_{F_{t+1}/F_t}$. Similarly,
one can write:
\begin{equation}
	H^2(F_t, M) = M^{F_t}/N_{F_t} M, \label{eqn:H2_definition}
\end{equation}
with the restriction and corestriction maps given by the formulas:
\[
\begin{aligned}
	\res \colon H^2(F_{t+1}, M) &\longrightarrow H^2(F_t, M), \\
	[x] &\longmapsto [x], \\[10pt]
	\cores \colon H^2(F_t, M) &\longrightarrow H^2(F_{t+1}, M), \\
	[x] &\longmapsto [N_{F_{t+1}/F_t} \cdot x].
\end{aligned}
\]
Recall that $F_t$ acts on $H^i(F_t, M)$ trivially (cf. \cite[Proposition~3, p.~116]{Serre1979}). Moreover $p^t = |F_t|$ kills $H^i(F_t, M)$. Thus $H^i(F_t, M)$ becomes a $W_t[F/F_t]$-module.
\begin{Lemma} \label{lem:Tor_ses}
	For any $M \in \Mod_{W[F]}^{\mathrm{ff}}$ there exists a short exact sequence:
	\[
	0 \to H^1(F, M) \otimes_W k \to H^1(F, M \otimes_W k) \to H^2(F, M)[p] \to 0.
	\]
\end{Lemma}
\begin{proof}
	The universal coefficient theorem for the group homology (cf. \cite[\S 3.6]{Weibel}) implies that for any $i > 0$
	\[
	0 \to H_i(F, M) \otimes_W k \to H_i(F, M \otimes_W k) \to \Tor^1_W(H_{i-1}(F, M), k) \to 0.
	\]
	Moreover, $\Tor^1_W(H_{i-1}(F, M), k) \cong H_{i-1}(F, M)[p]$.
	The proof follows by~\eqref{eqn:periodicity}.
\end{proof}
\subsection{$F$-Yakovlev diagrams}
It turns out that knowing the first group cohomologies of a $W[F]$-module along with the restriction and corestriction maps is enough to recover it up to certain well-understood summands.
\begin{Definition}
	An \bb{$F$-Yakovlev diagram over $W$} is a diagram of the form:
	\begin{equation} \label{eqn:yd}
		\begin{tikzcd}
			\mc D_{1} \arrow[r, "\alpha_{1}", shift left] & \mc D_{2} \arrow[l, "\beta_{1}", shift left] \arrow[r, "\alpha_{2}", shift left] & \cdots \arrow[l, "\beta_{2}", shift left] \arrow[r, "\alpha_{n-2}", shift left] & \mc D_{n-1} \arrow[l, "\beta_{n-2}", shift left] \arrow[r, "\alpha_{n-1}", shift left] & \mc D_{n} \arrow[l, "\beta_{n-1}", shift left]
		\end{tikzcd}
	\end{equation}
	where $\mc D_t$ is a $W_t[F/F_t]$-module of finite $W$-length and $\alpha_t, \beta_t$
	are homomorphisms of $W[F]$-modules satisfying:
	\begin{equation} \label{eqn:alpha_beta_and_beta_alpha}
		\alpha_t \circ \beta_t = p, \qquad \beta_t \circ \alpha_t = N_{F_{t+1}/F_t}.
	\end{equation}
\end{Definition}
Let $\mc D = (\mc D_t, \alpha_t, \beta_t)$ and $\mc D' = (\mc D_t', \alpha_t', \beta_t')$
be $F$-Yakovlev diagrams over~$W$. A morphism $\Phi : \mc D \to \mc D'$ is a collection of $W[F]$-homomorphisms $\Phi_i : \mc D_i \to \mc D_i'$ for $1 \le i \le n$ such that $\Phi_{t+1} \circ \alpha_t = \alpha_t' \circ \Phi_t$ and $\Phi_t \circ \beta_t = \beta_t' \circ \Phi_{t+1}$ for any $1 \le t < n$.
Therefore one may define the category $\YD_W(F)$ of $F$-Yakovlev diagrams over $W$.
One easily checks that $\YD_W(F)$ is abelian and that every object of $\YD_W(F)$
is of finite length. Hence, by \cite[Lemma~5.1, Theorem~5.5]{Krause_KS_categories}, $\YD_W(F)$ is a Krull--Schmidt category (i.e. every object has a unique decomposition into a direct sum of indecomposable objects).\\

By the previous discussion, for any $W[F]$-module $A$ of finite $W$-length the diagram $\Delta(A)$
\begin{equation*}
	\begin{tikzcd} 
		H^1(F_1, A) \arrow[r, "\cores", shift left] & H^1(F_2, A) \arrow[l, "\res", shift left] \arrow[r, "\cores", shift left] & \cdots \arrow[l, "\res", shift left] \arrow[r, "\cores", shift left] & H^1(F_n, A) \arrow[l, "\res", shift left]
	\end{tikzcd}
\end{equation*}
is an $F$-Yakovlev diagram over $W$. Clearly, in this way we obtain a functor $\Delta : \Mod_{W[F]} \to \YD_W(F)$. The following result is basically \cite[Theorem~4]{Yakovlev_basis_powers}. We discuss it in the next section.
\begin{Theorem} \label{thm:yako1}
	If $M_1$, $M_2$ are objects in $\Mod_{W[F]}^{\mathrm{ff}}$ satisfying $\Delta(M_1) \cong \Delta(M_2)$
	then for some $a_i, b_i \ge 0$:
	\[
		M_1 \oplus \bigoplus_{i = 0}^n W[F/F_i]^{a_i}
		\cong M_2 \oplus \bigoplus_{i = 0}^n W[F/F_i]^{b_i}.
	\]
\end{Theorem}
\subsection{Operations on Yakovlev diagrams}
In the sequel we use the following notation for an $F$-Yakovlev diagram $(\mc D_t, \alpha_t, \beta_t)$ and $1 \le i \le j \le n$:
\begin{align*}
	\alpha_{i, j} := \alpha_{j-1} \circ \alpha_{j-2} \circ \cdots \circ \alpha_i : \mc D_i \to \mc D_j,\\
	\beta_{i, j} := \beta_i \circ \beta_{i+1} \circ \cdots \circ \beta_{j-1} : \mc D_j \to \mc D_i
\end{align*}
(in particular, $\alpha_{i, i} = \beta_{i, i} = \id$).
We now discuss several operations on Yakovlev diagrams. Firstly,
note that there is an obvious notion of a direct sum of Yakovlev diagrams,
given by summing the terms and maps of the sequences.
For any $\mc D \in \YD_W(F)$ we write $\mc D \otimes k$ for the Yakovlev diagram
\[
(\mc D \otimes k)_t := \mc D_t \otimes k
\]
with the maps induced by the maps on $\mc D$. 
For any $1 \le i \le n$ we define also the 
$F$-Yakovlev diagram $\mc D^{(i)}$ as follows:
\begin{equation*}
	\mc D^{(i)}_t := 
	\begin{cases}
		\mc D_t, & t \le i,\\
		\im(\alpha_{i, t} : \mc D_i \to \mc D_t), & t > i,
	\end{cases}
\end{equation*}
with the transition maps $\alpha_t|_{\mc D^{(i)}_t}$ and $\beta_t|_{\mc D^{(i)}_{t+1}}$. We also define $\mc D^{(0)}$ to be the zero Yakovlev diagram. The next two operations come naturally
from standard operations on $W[F]$-modules. For any $\mc D \in \YD_W(F_T)$ define the induced 
Yakovlev diagram $\Ind^{F}_{F_T} \mc D \in \YD_W(F)$ as follows:
\[
(\Ind^{F}_{F_T} \mc D)_t :=
\begin{cases}
	\Ind^{F}_{F_T} \mc D_t, & t \le T,\\
	\Ind^{F/F_t} \mc D_T, & t \ge T,
\end{cases}
\]
(note that for $t = T$ both descriptions are equal).
We define the transition maps in $\Ind^{F}_{F_T} \mc D$ as follows:
\begin{itemize}
	\item for $t \le T-1$ the maps
	\[
	\Ind^{F}_{F_T} \mc D_t \leftrightarrows \Ind^{F}_{F_T} \mc D_{t+1}
	\]
	are the maps induced from $\mc D_t \leftrightarrows \mc D_{t+1}$,
	
	\item for $t \ge T$ the maps
	\[
	\Ind^{F/F_t} \mc D_T \leftrightarrows \Ind^{F/F_{t+1}} \mc D_T
	\]
	are of the form $q_{G, H}$ and $\rho_{G, H}$ for $G = F/F_t$, $H = F_{t+1}/F_t$.
\end{itemize}
\begin{Lemma} \label{lem:induced_YD}
	For any $M \in \Mod_{W[F_T]}$ and any $i \in \{ 1, 2 \}$:
	\begin{equation*}
		H^i(F_t, \Ind^{F}_{F_T} M)
		\cong
		\begin{cases}
			\Ind^{F}_{F_T} H^i(F_t, M), & t \le T,\\
			\Ind^{F/F_t} H^i(F_T, M), & t > T.
		\end{cases}
	\end{equation*}
	In particular, for any $M \in \Mod_{W[F_T]}$ we have an isomorphism of $F$-Yakovlev diagrams $\Ind^{F}_{F_T} (\Delta(M)) \cong \Delta(\Ind^{F}_{F_T} M)$.
\end{Lemma}
\begin{proof}
	Note that for any finite abelian group $G$, a $G'$-module $M$ and any subgroups $G'' \subset G' \subset G$ we have an isomorphism of $G$-modules:
	\begin{equation} \label{eqn:cohomology_res_ind}
		H^i(G'', \Ind^G_{G'} M) \cong \Ind^G_{G'} H^i(G'', M).
	\end{equation}
	Indeed, $\Ind^G_{G'} M \cong \bigoplus_{G/G'} M$ and hence $H^i(G'', \Ind^G_{G'} M) \cong \bigoplus_{G/G'} H^i(G'', M)$. One easily checks that
	the $G$-action on the latter direct sum is the induced action.
	Hence, by~\eqref{eqn:cohomology_res_ind} with $G := F$, $G' := F_T$ and $G'' := F_t$, where $t \le T$, we obtain:
	\begin{align*}
		H^i(F_t, \Ind^{F}_{F_T} M) \cong \Ind^{F}_{F_T} H^i(F_t, M).
	\end{align*}
	If $t > T$ then, using~\eqref{eqn:cohomology_res_ind} for $G = F$, $G' = G'' = F_t$
	and Shapiro's lemma (cf. \cite[Lemma~6.3.2]{Weibel}):
	\begin{align*}
		H^i(F_t, \Ind^{F}_{F_T} M) &
		\cong H^i(F_t, \Ind^{F}_{F_t} \Ind^{F_t}_{F_T} M)\\
		&\cong \Ind^{F/F_t} H^i(F_t, \Ind^{F_t}_{F_T} M)\\
		&\cong \Ind^{F/F_t} H^i(F_T, M). \qedhere
	\end{align*}
\end{proof}
\noindent 
For any $\mc D \in \YD_W(F)$ we define the Yakovlev diagram $\Fix_{F_T}(\mc D) \in \YD_W(F/F_T)$~as:
\[
\Fix_{F_T}(\mc D)_t := \ker(\beta_{T, t+T} : \mc D_{t+T} \to \mc D_T)
\]
for $t = 1, \ldots, n - T$. The transition maps in $\Fix_{F_T}(\mc D)$ are restrictions of the maps $\alpha_{t+T}, \beta_{t+T}$
to the modules $\Fix_{F_T}(\mc D)_t \subset \mc D_{t+T}$. Note that for any $t \ge 1$ we have:
\begin{equation} \label{eqn:pT_Delta_subset_Fix_Delta}
	p^T \cdot \mc D_{t+T} \subset \Fix_{F_T}(\mc D)_t \subset \mc D_{t+T}.
\end{equation}
Indeed, this follows from the fact that $\mc D_T$ is a $W_T$-module and
hence the image of $p^T \cdot \mc D_{t+T}$ in $\mc D_T$ is zero.
\begin{Lemma} \label{lem:operations_on_yd}
For any $M \in \Mod_{W[F]}$ we have an isomorphism of $F/F_T$-Yakovlev diagrams $\Fix_{F_T}(\Delta(M)) \cong \Delta(M^{F_T})$.
\end{Lemma}
\begin{proof}
	Let $\mc D := \Delta(M)$. By the inflation--restriction exact sequence:
	\begin{align*}
		H^1(F_{t+T}/F_T, M^{F_T}) &= \ker(H^1(F_{t+T}, M) \to H^1(F_T, M))\\
		&= \ker(\beta_{T, t+T} : \mc D_{t+T} \to \mc D_T)\\
		&= \Fix_{F_T}(\mc D)_t.
	\end{align*}
	This ends the proof.
\end{proof}
\subsection{$G$-Yakovlev diagrams}
Assume now that $G$ is a finite group with~$F$ as a $p$-Sylow subgroup.
Write $N_i := N_G(F_i)$ for the normalizer of $F_i$ in $G$.
\begin{Definition}
	A \bb{$G$-Yakovlev diagram over $W$} is a diagram of the form~\eqref{eqn:yd},
	where $\mc D_t$ is a $W_t[N_t/F_t]$-module that is finitely generated as a $W$-module and $\alpha_t, \beta_t$
	are homomorphisms of $W[N_{t+1}]$-modules satisfying the equations~\eqref{eqn:alpha_beta_and_beta_alpha} for all $1 {\le} t {<} n$. 
\end{Definition}
One defines morphisms between $G$-Yakovlev diagrams in a similar manner as above. We write $\YD_W(G)$ for the category of $G$-Yakovlev diagrams over $W$.
Clearly, $\Delta$ defined as above extends to a functor $\Delta : \Mod_{W[G]} \to \YD_W(G)$.
Recall that a \bb{permutation module} for $G$ is a module of the form $\Ind^G_H W$
for some subgroup $H$. A \bb{trivial source module} for $G$ is any direct summand 
of a permutation module. The following result is essentially Theorem 2.1 of \cite{Yakovlev_cyclic_sylow}. We discuss it further in the next section.
\begin{Theorem} \label{thm:yako2}
	The functor $\Delta : \Mod_{W[G]}^{\mathrm{ff}} \to \YD_W(G)$ is essentially surjective. Moreover, if $M$, $N$ are objects in $\Mod_{W[G]}^{\mathrm{ff}}$ satisfying $\Delta(M) \cong \Delta(N)$
	then there exist trivial source $W[G]$-modules $M^0, N^0$ such that:
	\[
	M \oplus N^0 \cong N \oplus M^0.
	\]
\end{Theorem}
In the sequel we will need also the following result concerning $W[G]$-modules.
Recall that a finite group is called $p$-hypo-elementary,
if it is of the form $P \rtimes C$, where $P$ is a normal $p$-subgroup
and $C$ is a cyclic group, $p \nmid |C|$. Write $\mc H_G$ for the set
of $p$-hypo-elementary subgroups of $G$.
\begin{Lemma}[Conlon's induction theorem] \label{lem:conlon}
	Suppose $M$ is a finitely generated $W[G]$-module. Then the isomorphism class of $M$
	is uniquely determined by all the restrictions of $M$ to $H \in \mc H_G$.
\end{Lemma}
\begin{proof}
	The proof is analogous to the proof of \cite[Lemma~3.2]{Bleher_Chinburg_Kontogeorgis_Galois_structure}.
	Write $W_H$ for the trivial $W[H]$-module. By \cite[(81.31)]{Curtis_Reiner_Methods_II},
	we have the following identity in the Green ring of $W[G]$-modules:
	\[
	[W_G] = \sum_{H \in \mc H_G} \alpha_H \cdot [\Ind^G_H W_H],
	\]
	where $\alpha_H \in \QQ$ are certain rational numbers. Tensoring this identity with~$M$ over~$W$ we obtain:
	\[
	[M] = \sum_{H \in \mc H_G} \alpha_H \cdot [\Ind^G_H \Res^G_H(M)],
	\]
	which ends the proof.
\end{proof}
\section{Yakovlev diagrams over Witt vectors -- proofs} \label{sec:yd_witt}
Theorems~\ref{thm:yako1} and~\ref{thm:yako2} were originally stated in the case
$k = \FF_p$, $W = \ZZ_p$, but they remain true for an arbitrary perfect field $k$ of
characteristic~$p$ with ring of Witt vectors $W$. The proof
of Theorem~\ref{thm:yako1} goes through verbatim (see also \cite[Remark, p. 1012]{Yakovlev_basis_powers}) for general perfect~$k$. For the reader's convenience
we include a sketch of the proof. The argument behind Theorem~\ref{thm:yako2} uses the finiteness of $k$ at only one point; we indicate this point explicitly and explain how the proof should be modified there. This section may be skipped on a first reading.
\subsection{Proof of Theorem~\ref{thm:yako1}}
Let $\Lambda := W[F]$, $\Lambda_{i, t} := \frac{(1 - \sigma_i) W[F]}{(1 - \sigma_t) W[F]}$ for any $t \le i \le n$ and $\Lambda_{n+1, t} = W[F]/(1 - \sigma_t) W[F] \cong W[F/F_t]$. For any $W[F]$-module~$A$ write $A^{(t)} := \ker (N_{F_t} : A \to A)$ and $A^{(n+1)} := A$. Note that $N_{F_{t+1}/F_t} \cdot A^{(t+1)} \subset A^{(t)} \subset A^{(t+1)}$ for any $t \ge 1$.\\

Suppose that $(\mc D, \alpha_t, \beta_t)$ is a given Yakovlev diagram. 
We want to find a $W[F]$-module $B$ with $\Delta(B) \cong \mc D$. 
The module~$B$ we construct will be ``minimal'' in the sense that it injects into any other module $C$ with $\Delta(C) \cong \mc D$. We will construct, by induction on~$i$, modules~$B_i$ which will eventually be identified with the submodules~$B^{(i)}$ of the desired $W[F]$-module~$B$. Simultaneously, we construct maps
$\pi_i : B^{(i)} \to \mc D_i = H^1(F_i, B) = B^{(i)}/(\sigma_i - 1) \cdot B$ and $N_{F_{i+1}/F_i} : B^{(i+1)} \to B^{(i)}$. The image of $\pi_{i+1} \oplus N_{F_{i+1}/F_i} : B_{i+1} \to \mc D_{i+1} \oplus B^{(i)}$, denoted $\mc W_i$, should contain the ``old space'' $\mc V_i := ((\alpha_i \circ \pi_i) \oplus N_{F_{i+1}/F_i})(B^{(i)})$. 
In order to create $B^{(i+1)}$ from $B^{(i)}$ we add to $B^{(i)}$ new elements $x_1, \ldots, x_m$ whose images generate $\mc W_i/\mc V_i$. This idea is made precise in the following result.
\begin{Proposition}[{\cite[Theorem~1]{Yakovlev_basis_powers}}] \label{prop:YD_induction}
	For every $1 \le i \le n$ there exists a $W[F]$-module $B_i$ and epimorphisms $\pi_i^t : B_i^{(t)} \to \mc D_t$ for $1 \le t \le i$ such that:
	\begin{enumerate}
		\item the following diagram is commutative:
		\begin{center}
			\begin{tikzcd}
				B_i^{(1)} \arrow[d, "\pi^1_i"] & B_i^{(2)} \arrow[d, "\pi^2_i"] \arrow[l, "N_{F_2/F_1}" above] & \ldots \arrow[l, "N_{F_3/F_2}" above]            & B_i^{(i)} \arrow[d, "\pi^i_i"] \arrow[l, "N_{F_i/F_{i-1}}" above] \\
				\mc D_1                             & \mc D_2 \arrow[l, "{\beta_{ 1}}"]                  & \ldots \arrow[l, "{\beta_{ 2}}"] & \mc D_i, \arrow[l, "{\beta_{ i-1}}"]                     
			\end{tikzcd}
		\end{center}

		\item $\pi_i^{t+1}|_{B_i^{(t)}} = \alpha_t \circ \pi_i^t$ for any $i > t > 0$,
		\item $N_{F_{t+1}/F_t} \cdot B_i^{(t+1)} = (\pi_i^t)^{-1}(\im \beta_t)$ for any $i \ge t > 0$,
		\item $N_{F_{t+1}/F_t} \cdot \ker \pi^{t+1}_i = \ker \pi_i^t $ for any $i > t > 0$,
		\item $B_i^{(i)} = B_i$,
		\item $B_i$ is a free finitely generated $W$-module,
		\item If $C$ and $\delta^t : C^{(t)} \to \mc D_t$ are any other module and epimorphisms with the above properties,
		then there exists a monomorphism\\$\gamma : B_i \to C$ which induces an isomorphism
		$C \cong B_i \oplus \bigoplus_{t \le i} \Lambda_{i, t}^{a_{it}}$ for some $a_{it} \in \ZZ_{\ge 0}$ such that $(\delta^t \circ \gamma) |_{B_i^{(t)}} = \pi_i^t$.

	\end{enumerate}
\end{Proposition}
In order to construct $B_i$, one needs the following notion. For any $\Lambda$-module
$A$ with a submodule $Q$, a \bb{minimal system of generators of $A$ relative to $Q$} is
a tuple $a_1, \ldots, a_N$ such that $A = \Lambda a_1 + \cdots + \Lambda a_N + Q$
and $N$ is smallest possible (cf. \cite[\S 3]{Yakovlev_basis_powers}).
\begin{proof}[Sketch of proof of Proposition~\ref{prop:YD_induction}]
	We proceed by induction on $i$. Define $B_0 := 0$. The inductive step proceeds as follows. Having constructed $B_i$, let:
	\begin{align*}
		\mc W_i &:= \mc D_{i+1} \times_{\mc D_i} B_i = \{ (a, b) \in \mc D_{i+1} \oplus B_i : \beta_i(a) = \pi^i_i(b) \},\\
		\mc V_i &:= \im \left((\alpha_i \circ \pi^i_i) \oplus N_{F_{i+1}/F_i} : B_i \to \mc D_{i+1} \oplus B_i \right).
	\end{align*}
	Write $e_1, \ldots, e_m$ for a minimal system of generators of $\mc W_i$ relative to $\mc V_i$ and let:
	\[
	B_{i+1} := \frac{B_i \oplus \bigoplus_{l = 1}^m \Lambda \cdot x_l}{\langle N_{F_{i+1}/F_i} \cdot x_l - \pr_2(e_l) : l = 1, \ldots, m \rangle},
	\]
	where $\bigoplus_{l = 1}^m \Lambda \cdot x_l$ is a free $\Lambda$-module with generators $x_1, \ldots, x_m$ and $\pr_2 : \mc D_{i+1} \oplus B_i \to B_i$ is the projection onto the second component. For $t \le i$ one easily shows that $B_{i+1}^{(t)} = B_i^{(t)}$. Thus we may put $\pi_{i+1}^t := \pi_i^t$. For $t = i+1$ we put:
	\begin{align*}
		\pi_{i+1}^{i+1}(b) := \alpha_i \circ \pi_i^i(b) \quad \textrm{ for } b \in B_{i+1}^{(i)} = B_i,\\
		\pi_{i+1}^{i+1}(x_l) := \pr_1(e_l) \qquad \textrm{ for } l = 1, \ldots, m.
	\end{align*}
	One proves the properties (1) -- (7) precisely as in \cite[\S 4]{Yakovlev_basis_powers},
	using only the fact that $\Lambda/N_{F_{i+1}/F_i} \cdot \Lambda \cong W[\zeta_{p^{n-i}}]$ is a local ring, which is true also in our case. This fact is used twice: firstly in the proof of the property~(6),
	and then in the proof of \cite[Lemma~3]{Yakovlev_basis_powers}, which is used in~(7). 
\end{proof}
Theorem~\ref{thm:yako1} follows directly from Proposition~\ref{prop:YD_induction} by defining $B := B_n$.
\begin{Corollary}[{\cite[p. 218]{Yakovlev_cyclic_sylow}}] \label{cor:existence_of_F_morphism}
	Suppose that $M, N \in \Mod^{\mathrm{ff}}_{W[G]}$ and that there exists an
	isomorphism $\xi : \Delta(M) \to \Delta(N)$ of Yakovlev diagrams. Then
	there exist $W[F]$-homomorphisms $\Xi : M \to N$ and $\Xi' : N \to M$, which induce
	$\xi$ and $\xi^{-1}$ respectively.
\end{Corollary}
\begin{proof}
	Let $B$ be the $W[F]$-module satisfying $\Delta(B) \cong \Delta(M) \cong \Delta(N)$,
	constructed in Proposition~\ref{prop:YD_induction}. Then by Proposition~\ref{prop:YD_induction}~(7) there exist split monomorphisms $\gamma : B \to M$
	and $\gamma' : B \to N$ that induce isomorphisms $\Delta(B) \cong \Delta(M) \cong \Delta(N)$. Let $\varphi : M \to B$, $\varphi' : N \to B$ be the left inverses
	of $\gamma$ and $\gamma'$. Then one can define $\Xi := \gamma' \circ \varphi$
	and $\Xi' := \gamma \circ \varphi'$.
\end{proof}

\subsection{Proof of Theorem~\ref{thm:yako2}}
We start by sketching the existence part of Theorem~\ref{thm:yako2}.
\begin{proof}[Sketch of proof of Theorem~\ref{thm:yako2} (the existence part), cf. {\cite[\S 4]{Yakovlev_cyclic_sylow}}] \mbox{}\\
	One proves that for each $\mc D \in \YD_W(G)$ there exists a $W[G]$-module $M$
	with $\Delta(M) \cong \mc D$ in a manner similar to the proof of Proposition~\ref{prop:YD_induction}. Namely, one constructs the
	sequences $(B_i)_i$, $(\mc V_i)_i$, $(\mc W_i)_i$ of $W[G]$-modules by the same formulas with $\Lambda := W[G]$,
	where $e_1, \ldots, e_m \in \mc W_i$ are any elements that generate $\mc W_i/\mc V_i$ over $W[N_{i+1}]$.
\end{proof}
\begin{Proposition}[{\cite[Theorem~2.4]{Yakovlev_cyclic_sylow}}] \label{prop:trivial_source}
	Assume that an indecomposable $A \in \Mod^{\mathrm{ff}}_{W[G]}$ satisfies $\Delta(A) = 0$. Then $A$ is isomorphic to a direct summand of $W[G/F_i]$ for some $i$.
\end{Proposition} 
\begin{proof}
	Let $\Lambda := W[G]$, $\Gamma := W[F]$ and $B := \Lambda \otimes_{\Gamma} A$. One easily checks that the mapping $\varphi : B \to A$, $\varphi(x \otimes a) = xa$ is a $\Lambda$-homomorphism. 
	Let $G = \bigcup_{j = 1}^m F g_j F$ be the double coset decomposition, where $g_1 = e$. 
	Note that $B$ decomposes as
	\[
	\bigoplus_{j = 1}^m (\Gamma g_j \Gamma) \otimes_{\Gamma} A
	\]
	and the restriction of $\varphi$ to the first summand is the isomorphism $\Gamma \otimes_{\Gamma} A \to A$. Thus $B \cong A \oplus X$ as $W[F]$-modules (where $X := \ker \varphi$).
	Since an exact sequence of $W[G]$-modules splits if and only if it splits as a sequence
	of $W[F]$-modules (cf. \cite[XV.7~(6) and the definition in~XII.2]{Cartan_Eilenberg_Homological_algebra}) we obtain $B \cong A \oplus X$ as $W[G]$-modules.
	On the other hand, as a $W[F]$-module, $A$ is a direct sum of modules
	of the form $W[F/F_i]$ by \cite{Yakovlev_basis_powers}. Thus, as a $W[G]$-module,
	$B$ is a direct sum of the modules $W[G] \otimes_{W[F]} W[F/F_i] \cong W[G/F_i]$.
	This finishes the proof, since the decomposition into indecomposable direct summands is unique.
\end{proof}
\begin{proof}[Sketch of proof of Theorem~\ref{thm:yako2} (the uniqueness part), cf. {\cite[\S 5]{Yakovlev_cyclic_sylow}}]
	Assume that $M$ and $N$ are two $W[G]$-modules such that there exists an isomorphism $\xi : \Delta(M) \to \Delta(N)$ of $G$-Yakovlev diagrams. By Corollary~\ref{cor:existence_of_F_morphism}
	there exist $W[F]$-homomorphisms $\Xi : M \to N$, $\Xi' : N \to M$ which induce
	$\xi$ and $\xi^{-1}$. Let $g_1, \ldots, g_m \in G$ be the representatives of $G/F$.
	Define the $W[G]$-homomorphisms $f : M \to N$ and $g : N \to M$ as:
	\begin{align*}
		f(x) := \frac 1m \sum_{i = 1}^m g_i^{-1} \cdot \Xi(g_i \cdot x),\\
		g(x) := \frac 1m \sum_{i = 1}^m g_i^{-1} \cdot \Xi'(g_i \cdot x).
	\end{align*}
	The maps $f$ and $g$ still induce $\xi$ and $\xi^{-1}$ on $\Delta(M)$ and $\Delta(N)$, cf. \cite[Lemma~5.1]{Yakovlev_cyclic_sylow}.
	Pick $f'$ and $g'$ as in Proposition~\ref{prop:key_lemma_G_yakovlev} below. Define $M^0 := \ker f'$, $M^{\Delta} := \im g'$, $N^0 := \ker g'$, $N^{\Delta} := \im f'$. One easily checks that $M \cong M^{\Delta} \oplus M^0$, $N \cong N^{\Delta} \oplus N^0$, $\Delta(M^0) \cong \Delta(N^0) \cong 0$ and $f'$, $g'$ induce isomorphisms between $M^{\Delta}$ and $N^{\Delta}$. The proof follows by Proposition~\ref{prop:trivial_source}.
\end{proof}
\begin{Proposition} \label{prop:key_lemma_G_yakovlev}
	Let $M$ and $N$ be two $W[G]$-modules that are free and finitely generated as $W$-modules. Suppose that $f : M \to N$ and $g : N \to M$ are two
	$W[G]$-homomorphisms such that $\Delta(f) \circ \Delta(g) = \id$,
	$\Delta(g) \circ \Delta(f) = \id$. Then there exist $W[G]$-homomorphisms $f' : M \to N$ and $g' : N \to M$ such that:
	\begin{enumerate}
		\item $f' \circ g' \circ f' = f'$,
		\item $g' \circ f' \circ g' = g'$,
		\item $\Delta(f') \circ \Delta(g') = \id$ and $\Delta(g') \circ \Delta(f') = \id$.
	\end{enumerate}
\end{Proposition}
The original proof of Proposition~\ref{prop:key_lemma_G_yakovlev} (cf. \cite[Lemma~5.2]{Yakovlev_cyclic_sylow})
uses the fact that $k$ is finite. Thus this is the only part in which we 
diverge from Yakovlev's work. The main idea is to define
$f'$ and $g'$ as $f \circ F(g \circ f)$ and $g \circ G(f \circ g)$
for certain polynomials $F, G \in W[x]$. In the sequel we will need
the following fact.
\begin{Lemma} \label{lem:inverse}
	For any monic polynomials
	$A, B \in W[x]$:
	\begin{equation*}
		A(x) \in \left( W[x]/(B(x)) \right)^{\times} \quad \Leftrightarrow \quad
		\GCD(\ol A, \ol B) = 1 \quad \textrm{ in } k[x],
	\end{equation*}
	where $\ol A$ and $\ol B$ denote the images of $A, B \in W[x]$ in $k[x]$.
\end{Lemma}
\begin{proof}
	Clearly, if $A(x) \in \left( W[x]/(B(x)) \right)^{\times}$ then $A \cdot S + B \cdot T = 1$ for some $S, T \in W[x]$.
	Thus $\ol A \cdot \ol S + \ol B \cdot \ol T = 1$, which implies that $\GCD(\ol A, \ol B) = 1$.\\
	The Euclidean algorithm implies that $\GCD(A, B) \in W[x]$ is a monic polynomial. Moreover, $\ol{\GCD(A, B)} \mid \GCD(\ol A, \ol B)$.
	Thus $\GCD(\ol A, \ol B) = 1$ implies that $\GCD(A, B) = 1$ in $K[x]$. In particular, $A \cdot S + B \cdot T = 1$ for some $S, T \in K[x]$,
	$\deg S < \deg B$, $\deg T < \deg A$, $(S, T) \neq (0, 0)$. Let $\alpha \in \ZZ$ be the maximal number such that
	$S = p^{\alpha} \cdot S_1$, $T = p^{\alpha} \cdot T_1$, $S_1, T_1 \in W[x]$. Clearly, $\alpha \le 0$. Assume to the contrary that $\alpha \neq 0$.
	Then $\ol A \cdot \ol S_1 + \ol B \cdot \ol T_1 = \ol{p^{- \alpha}} = 0$ and $(\ol S_1, \ol T_1) \neq (0, 0)$. This easily implies that $\GCD(\ol A, \ol B) \neq 1$.
	The contradiction means that $\alpha = 0$ and $S, T \in W[x]$. Hence $S$ is the inverse of $A$ in $W[x]/(B(x))$.
\end{proof}

\begin{proof}[Proof of Proposition~\ref{prop:key_lemma_G_yakovlev}]
	For any endomorphism $\varphi$ we write $\chi_{\varphi}(x) := \det(x \cdot \id - \varphi)$
	for its characteristic polynomial. By \cite[\S 11.6, Corollary]{Prasolov_linear_algebra} we have $\chi_{f \circ g}(x) = x^c \cdot \chi_{g \circ f}(x)$ for $c := \rank_W N - \rank_W M \in \ZZ$. Without loss of generality, assume that $c \le 0$. By the main theorem on Newton polygons (see e.g. \cite[Proposition II.6.4]{Neukirch_ANT}) we can factor $\chi_{g \circ f}(x)$ as:
	\[
	P(x) \cdot Q(x),
	\]
	where $P, Q \in W[x]$, $p \nmid P(0)$ and all the roots of $Q(x)$ in $\ol K$ have positive valuation. In particular, all the coefficients of $Q(x)$ except the leading one have positive valuation and $Q(1) \in W^{\times}$. By Lemma~\ref{lem:inverse}, since $p \nmid P(0)$, the element $x \cdot Q(x)^2$ has an inverse in $W[x]/((x-1) \cdot P(x))$. We fix a representative $S(x) \in W[x]$ of this inverse. 
	Let:
	\begin{align*}
		F(x) &:= S(x) \cdot Q(x) \cdot Q(1),\\
		G(x) &:= Q(x) \cdot Q(1)^{-1},\\
		f' &:= f \circ F(g \circ f),\\
		g' &:= g \circ G(f \circ g).
	\end{align*}
	Clearly, $f'$ and $g'$ are homomorphisms of $W[G]$-modules.
	Note now that
	\begin{equation} \label{eqn:chi|xFFG-F}
		\chi_{g \circ f}(x) \mid x \cdot F(x)^2 \cdot G(x) - F(x)
	\end{equation}
	in $W[x]$. Indeed:
	\[
	\frac{x \cdot F(x)^2 \cdot G(x) - F(x)}{\chi_{g \circ f}(x)} = S(x) \cdot Q(1) \cdot \frac{x \cdot Q(x)^2 \cdot S(x) - 1}{P(x)}, 
	\]
	and the last fraction is in $W[x]$ by assumption. On the other hand, by the Cayley--Hamilton theorem (cf. \cite[\S 13.2]{Prasolov_linear_algebra}) we have $\chi_{g \circ f}(g \circ f) = 0$. Therefore, by~\eqref{eqn:chi|xFFG-F}, using the identity
	$g \circ (f \circ g)^i = (g \circ f)^i \circ g$:
	\begin{align*}
		f' \circ g' \circ f' &= f \circ F(g \circ f) \circ g \circ G(f \circ g)
		\circ f \circ F(g \circ f)\\
		&= f \circ (g \circ f) \circ F(g \circ f)^2 \circ G(g \circ f)\\
		&= f \circ F(g \circ f) = f'.
	\end{align*}
	Similarly, one obtains $g' \circ f' \circ g' = g'$. We prove now that $f'$ and $g'$ induce mutually inverse isomorphisms on $\Delta(M)$, $\Delta(N)$. Indeed, the definition of $S$ easily implies that $S(1) = Q(1)^{-2}$ and thus $F(1) = G(1) = 1$. Hence:
	\begin{align*}
		\Delta(f') \circ \Delta(g') &= \Delta(f) \circ F(\Delta(g \circ f)) \circ \Delta(g) \circ G(\Delta(f \circ g))\\
		&= \Delta(f) \circ F(\id) \circ \Delta(g) \circ G(\id) = \Delta(f) \circ \Delta(g) = \id,
	\end{align*}
	and similarly $\Delta(g') \circ \Delta(f') = \id$. This finishes the proof.
\end{proof}
\section{The modules $\JJ_i$ and $\SS_i$} \label{sec:J_and_I}
In this section we discuss various properties of the $W[F]$-modules $\JJ_i$
and $\SS_i$ that appear in the description of the crystalline cohomology of $\ZZ/p^n$-covers. We start by giving a more explicit description
of those modules. For any $a = \sum_{j = 0}^{p^n - 1} a_j \sigma^j \in W[F]$ we put:
\[
S_{r, i}(a) := \sum_{j \equiv r \pmod{p^{n - i}}} a_j \in W.
\]
One easily checks that the map $W[F] \to W[F/F_i]$ is given by the formula:
\[
	a \mapsto \sum_{t = 0}^{p^{n-i}-1} S_{t, i}(a) \cdot \ol{\sigma}^t,
\]
where $\ol{\sigma}$ is the image of $\sigma$ in $F/F_i$.
Hence the modules $\JJ_i$ and $\SS_i$ (defined as in Section~\ref{sec:intro})
can be explicitly described as:
\begin{align*}
	\JJ_i &= \{ a \in W[F] : S_{0, i}(a) = \ldots = S_{p^{n-i} - 1, i}(a) = 0 \},\\
	\SS_i &= \{ a \in W[F] : S_{0, i}(a) = \ldots = S_{p^{n-i} - 1, i}(a) \}.
\end{align*}
Note that $\JJ_i$ is the ideal in $W[F]$ generated by $\sigma_i - 1$.
It is immediate that $\SS_n \cong W[F]$ and $\SS_0 \cong W$ (a trivial $F$-module). Moreover,
$\JJ_n$ can be identified with $I_F$, the augmentation ideal of $W[F]$, and 
\begin{equation} \label{eqn:Ji_is_induced}
	\JJ_i \cong \Ind^F_{F_i} I_{F_i}.
\end{equation}
The modules $\JJ_i$ and $\SS_i$ are related through
the following exact sequence of $W[F]$-modules:
\begin{equation} \label{eqn:ses_for_SS}
	0 \to \JJ_i \to \SS_i \to W \to 0,
\end{equation}
where the first map is the inclusion and the second map is $a \mapsto S_{0, i}(a)$.\\

We now compute the Yakovlev diagrams of $\JJ_i$ and $\SS_i$.
Write $q_t$ for the quotient map $F/F_t \to F/F_{t+1}$
and $\rho_t$ for the map $W_i[F/F_{t+1}] \to W_i[F/F_t]$,
induced by $\rho_{G, H}$ with $G := F/F_t$, $H := F_{t+1}/F_t$.
Note also that for any $t$ there exists the multiplication-by-$p$ map
$W_t \to W_{t+1}$ (defined by $1 \mapsto p$), which we also denote by $p$ by abuse of notation. We write $r_t : W_{t+1} \to W_t$ for the reduction mod $p^t$ map.
\begin{Proposition} \label{prop:YD_Ji_Si}
	Fix $1 \le i \le n$ and let $\msm(t) := \min(i, t)$ and $\msM(t) := \max(i, t)$.
\begin{enumerate}
	\item We have $\Delta(\JJ_i) = (W_{\msm(t)}[F/F_{\msM(t)}], \alpha_{\JJ, t}, \beta_{\JJ, t})$, where:
	\begin{align*}
		\alpha_{\JJ, t} &=
		\begin{cases}
			p : W_t[F/F_i] \to W_{t+1}[F/F_i], & \textrm{ if } t < i,\\
			q_t : W_i[F/F_t] \to W_i[F/F_{t+1}], & \textrm{ if } t \ge i,
		\end{cases}\\
		\beta_{\JJ, t} &= 
		\begin{cases}
			r_t : W_{t+1}[F/F_i] \to W_t[F/F_i], & \textrm{ if } t < i,\\
			\rho_t : W_i[F/F_{t+1}] \to W_i[F/F_t],& \textrm{ if } t \ge i.
		\end{cases}
	\end{align*}
	
	\item We have $\Delta(\SS_i) = (W_{\msm(t)}[F/F_{\msM(t)}]/(N_{F/F_{\msM(t)}} \cdot W_{\msm(t)}), \alpha_{\SS, t}, \beta_{\SS, t})$, where the maps $\alpha_{\SS, t}, \beta_{\SS, t}$ are induced by the maps $\alpha_{\JJ, t}$, $\beta_{\JJ, t}$ defined above.
\end{enumerate}
\end{Proposition}
\begin{proof}
(1) By~\eqref{eqn:Ji_is_induced} and Lemma~\ref{lem:induced_YD} it suffices to prove that the Yakovlev diagram of~$I_F$ is isomorphic to the diagram:
\begin{equation*}
	\begin{tikzcd}
		W_{1} \arrow[r, "p", shift left] & W_{2} \arrow[l, "r_{1}", shift left] \arrow[r, "p", shift left] & \cdots \arrow[l, "r_{2}", shift left] \arrow[r, "p", shift left] & W_{n-1} \arrow[l, "r_{n-2}", shift left] \arrow[r, "p", shift left] & W_{n}, \arrow[l, "r_{n-1}", shift left]
	\end{tikzcd}
\end{equation*}
where $W_t$ has a trivial $F$-action.
We start by proving that
\begin{equation} \label{eqn:H1_of_Jn}
	H^1(F_t, \JJ_n) = \Span_{W_t}([\sigma_t - 1]) \cong W_t.
\end{equation}
Note that $\ker(N_{F_t} : \JJ_n \to \JJ_n) = \JJ_t = (\sigma_t - 1) W[F]$.
Moreover, for any $a \ge 0$ we have $[(\sigma_t - 1) \cdot \sigma^a] = [\sigma_t - 1]$, since: 
\[
(\sigma_t - 1) \cdot \sigma^a - (\sigma_t - 1) = (\sigma_t - 1) \cdot (\sigma^a - 1) \in (\sigma_t - 1) \cdot \JJ_n.
\]
This proves that $H^1(F_t, \JJ_n)$ is a trivial $F$-module and that $H^1(F_t, \JJ_n) = W \cdot [\sigma_t - 1]$. We show now that 
$p^j \cdot [\sigma_t - 1] = 0$ holds iff $j \ge t$. Indeed:
\begin{align*}
	p^t \cdot (\sigma_t - 1) &= p^t \cdot (\sigma_t - 1) - (\sigma_t^{p^t} - 1)\\
	&= (\sigma_t - 1) \cdot \sum_{l = 0}^{p^t - 1} ( 1 - \sigma_t^l) \in (\sigma_t - 1) \JJ_n.
\end{align*}
Finally, if $p^j \cdot [\sigma_t - 1] = 0$ then
\begin{align*}
	p^j \cdot (\sigma_t - 1) = (\sigma_t - 1) \cdot \sum_{l = 0}^{p^n - 1} a_l \sigma^l,
\end{align*}
where $a_l \in W$, $\sum_{l = 0}^{p^n - 1} a_l = 0$. For any $l \in \ZZ$ let $a_l := a_{l \bmod{p^n}}$. Then:
\[
p^j \cdot (\sigma_t - 1) = \sum_{l = 0}^{p^n - 1} (a_{l - p^{n-t}} - a_l) \cdot \sigma^l.
\]
Thus
\[
a_{l - p^{n-t}} - a_l =
\begin{cases}
	-p^j, & l = 0,\\
	p^j, & l = p^{n-t},\\
	0, & l \neq 0, p^{n-t}.
\end{cases}
\]
This easily implies that:
\[
a_l =
\begin{cases}
	a_{l \bmod{p^{n-t}}}, & p^{n-t} \nmid l,\\
	a_{p^{n-t}}, & p^{n-t} | l \textrm{ and } l \neq 0,\\
	a_{p^{n-t}} + p^j, & l = 0.
\end{cases}
\]
Therefore:
\[
0 = \sum_{l = 0}^{p^n - 1} a_l = p^t \cdot \sum_{l = 1}^{p^{n-t}} a_l + p^j,
\]
which yields $j \ge t$. This ends the proof of~\eqref{eqn:H1_of_Jn}. Finally, we have:
\begin{align*}
	\res([\sigma_{t+1} - 1]) &= (1 + \sigma_{t+1} + \cdots + \sigma_{t+1}^{p-1}) \cdot [\sigma_{t+1} - 1] = [\sigma_t - 1]\\
	\cores([\sigma_t - 1]) &= [\sigma_t - 1] = [\sigma_{t+1}^p - 1]
	= \sum_{l = 0}^{p - 1} \sigma_{t+1}^l \cdot [\sigma_{t+1} - 1]\\
	&= p \cdot [\sigma_{t+1} - 1],
\end{align*}
where for the last equality we used the fact that $F_{t+1}$ acts trivially on $H^1(F_{t+1}, \JJ_n)$. This ends the proof of (1).\\
(2) Observe that $H^1(F_t, W) = 0$. Thus the long exact sequence of
group cohomology applied to the exact sequence~\eqref{eqn:ses_for_SS} and (1) yield:
\[
H^1(F_t, \SS_i) = \coker(W \to H^1(F_t, \JJ_i) \cong W_{\msm(t)}[F/F_{\msM(t)}]).
\]
One easily checks that the map $W \to H^1(F_t, \JJ_i) \cong W_{\msm(t)}[F/F_{\msM(t)}]$ is
$a \mapsto N_{F/F_{\msM(t)}} \cdot a$.
\end{proof}
\begin{Corollary} \label{cor:betas_injective}
Suppose that $\DD \cong \Delta(\JJ_i)$ or $\DD \cong \Delta(\SS_i)$ for some $1 \le i \le n$. Then:
\begin{enumerate}
	\item $\ker(\beta_t \otimes k) = 0$ for all $1 \le t < n$,
	
	\item $\alpha_t : \DD_t \to \DD_{t+1}$ is surjective for all $t \ge i$.
\end{enumerate}
\end{Corollary}
\begin{proof}
	Assume that $\DD \cong \Delta(\SS_i)$. 
	If $t < i$, one easily checks that $\beta_t \otimes k = \id$. Assume now that $t \ge i$. Then $\beta_t \otimes k$ can be identified with:
	\[
		\rho_t : k[F/F_{t+1}]/(k \cdot N_{F/F_{t+1}}) \to k[F/F_t]/(k \cdot N_{F/F_t}).
	\]
	Note that if $a = \sum_{i = 0}^{p^{n-t-1} - 1} a_i \ol{\sigma}^i \in \DD_{t+1} \otimes k$
	then:
	\[
		\rho_t(a) = \sum_{i = 0}^{p^{n-t} - 1} a_{i \bmod p^{n - t - 1}} \cdot \ol{\sigma}^i.
	\]
	Hence, if $a \in \ker \rho_t$, then $a_0 = a_1 = \cdots = a_{p^{n-t-1} - 1}$
	and $a = 0$ in $k[F/F_{t+1}]/(k \cdot N_{F/F_{t+1}})$.
	Finally, the maps
	\[
		\alpha_{\SS, t} = q_t : W_i[F/F_t]/(W_i \cdot N_{F/F_t}) \to W_i[F/F_{t+1}]/(W_i \cdot N_{F/F_{t+1}})
	\]
	 are surjective for all $t \ge i$. One proves the case $\DD \cong \Delta(\JJ_i)$ in a similar manner.
\end{proof}
\begin{Corollary} \label{cor:Zp_modules}
	Let $F := \ZZ/p$. Up to isomorphism, there exist three indecomposable
	objects of $\Mod_{W[F]}^{\mathrm{ff}}$: $W$, $W[F]$ and $I_F$.
\end{Corollary}
\begin{proof}
	The $F$-Yakovlev diagrams correspond in this case to $k$-vector spaces of finite dimension. Thus the only indecomposable $F$-Yakovlev diagram is $\mc D = k$. Clearly, $\mc D = \Delta(I_F)$. By Theorem~\ref{thm:yako1}, the only indecomposable $W[F]$-modules with trivial $F$-Yakovlev diagrams are $W$ and $W[F]$.
\end{proof}
In the sequel we will also need the formula for the fixed points of $\JJ_i$ and $\SS_i$
with respect to $F_1$. Note that the map:
\[
	W[F'] \to W[F], \quad \sum_{j = 0}^{p^{n-1} - 1} a_j \cdot \ol{\sigma}^j
	\mapsto \sum_{j = 0}^{p^n - 1} a_{j \mod{p^{n-1}}} \cdot \sigma^j
\]
induces an $F'$-isomorphism onto $W[F]^{F_1}$. Indeed, if $a = \sum_{j = 0}^{p^n - 1} a_j \sigma^j \in W[F]^{F_1}$ then $a_j = a_{j + p^{n-1}}$ for $0 \le j < p^n - p^{n-1}$
by $F_1$-invariance. One easily checks that the above map induces the following isomorphisms of $W[F']$-modules:
\begin{equation} \label{eqn:fix_Ji_Si}
	\JJ_i^{F_1} \cong \JJ_{i-1}, \qquad \SS_i^{F_1} \cong \SS_{i-1}.
\end{equation}
For any $1 \le i \le p^n$ write $J_i$ for the unique indecomposable $k[F]$-module of dimension $i$ over $k$. Note that $J_i = (\sigma - 1)^{p^n - i} \cdot k[F]$.
\begin{Lemma} \label{lem:JJ_otimes_k}
	We have $\JJ_i \otimes_W k \cong J_{p^n - p^{n-i}}$ and $\SS_i \otimes_W k \cong J_{p^n - p^{n-i} + 1}$.
\end{Lemma}
\begin{proof}
	Clearly:
	\[
	\JJ_i \otimes_W k \cong (\sigma_i - 1) k[F] \cong (\sigma - 1)^{p^{n-i}} k[F] = J_{p^n - p^{n-i}}.
	\]
	Similarly:
	\begin{align*}
		\SS_i \otimes_W k &\cong \ker(k[F] \to k[F/F_i]/k)\\
		&= \ker(J_{p^n} \to J_{p^{n-i} - 1}) \cong J_{p^n - p^{n-i} + 1}. \qedhere
	\end{align*}
\end{proof}
\begin{Proposition} \label{prop:Delta_J_otimes_k}
	Fix $1 \le i \le n$.\\
(1) We have $H^2(F_t, \JJ_i) = 0$ for all $1 \le t \le n$ and $\Delta(\JJ_i \otimes k) \cong \Delta(\JJ_i) \otimes k$.\\
(2) We have $\Delta(\SS_i) \otimes k = \Delta(\SS_i \otimes k)^{(i)}$. Moreover
\[
	H^2(F_t, \SS_i)[p]
	\cong
	\begin{cases}
		0, & t \le i,\\
		k, & t > i.
	\end{cases}
\]
\end{Proposition}
\begin{proof}
	(1) By applying the long exact sequence of group cohomology to the short exact sequence:
	\[
		0 \to I_F \to W[F] \to W \to 0,
	\]
	one obtains $H^2(F_t, I_F) \cong H^1(F_t, W) \cong 0$ for all $0 \le t \le n$.
	Thus by~\eqref{eqn:Ji_is_induced} and Lemma~\ref{lem:induced_YD} we have
	$H^2(F_t, \JJ_i) = 0$ for all $0 \le i, t \le n$. Therefore the result follows by Lemma~\ref{lem:Tor_ses}.\\
	(2) By applying the long
	exact sequence of cohomology to~\eqref{eqn:ses_for_SS} and noting that $H^2(F_t, \JJ_i) = 0$ we obtain:
	\[
		H^2(F_t, \SS_i) = \ker(H^2(F_t, W) \to H^3(F_t, \JJ_i) \cong H^1(F_t, \JJ_i)).
	\]
	We have $H^2(F_t, W) \cong W_t$ and the map $W_t \to H^1(F_t, \JJ_i)$ is
	given by $a \mapsto [a \cdot N_{F/F_{\msM(t)}}]$ (cf. proof of Proposition~\ref{prop:YD_Ji_Si}). One easily checks that:
	\[
		\ker(N_{F/F_{\msM(t)}} : W_t \to W_{\msm(t)}[F/F_{\msM(t)}]) = 
		\begin{cases}
			0, & t \le i,\\
			p^i W_t\cong W_{t - i}, & t > i.
		\end{cases}
	\]
	For any $i > t$ consider the commutative diagram with exact rows (obtained from Lemma~\ref{lem:Tor_ses}):
	\begin{center}
		\begin{tikzcd}
			0 \arrow[r] & \Delta(\SS_i)_i \otimes k \arrow[d] \arrow[r] & \Delta(\SS_i \otimes k)_i \arrow[d] \arrow[r] & 0 \arrow[d] \arrow[r]          & 0 \\
			0 \arrow[r] & \Delta(\SS_i)_t \otimes k \arrow[r]           & \Delta(\SS_i \otimes k)_t \arrow[r]           & {H^2(F_t, \SS_i)[p]} \arrow[r] & 0
		\end{tikzcd}
	\end{center}
	Since by Corollary~\ref{cor:betas_injective}~(2) the first vertical arrow is surjective, $\Delta(\SS_i)_t \otimes k \cong \im(\Delta(\SS_i \otimes k)_i \to \Delta(\SS_i \otimes k)_t)$.
\end{proof}
\begin{Proposition} \label{prop:lifting_F_automorphisms}
	Assume that $\DD = \Delta(M)$ for $M$ of the form
	\begin{equation*}
		M := W[F]^a \oplus \bigoplus_{\nu = 1}^w \JJ_{b_{\nu}} \oplus \SS_m^b,
	\end{equation*}
	 with $m \ge \max(b_1, \ldots, b_w)$. Let $\mc D$ be an $F$-Yakovlev diagram over~$W$ and assume that there exists an isomorphism $\varphi : \DD \otimes k \to \mc D \otimes k$ of $F$-Yakovlev diagrams. Then $\varphi$	can be lifted to a morphism $\varphi' : \DD \to \mc D$ of $F$-Yakovlev diagrams.
\end{Proposition}
\begin{proof}
	Write $\DD^{\nu} := \Delta(\JJ_{b_{\nu}})$ for $1 \le \nu \le w$
	and $\DD^{\nu} := \Delta(\SS_m)$ for $w < \nu \le w + b$. Let $b_{\nu} := m$ for
	$w < \nu \le w + b$. Also, let $i_{t, \nu}(\sigma)$ be the image of $\sigma$ in
	$\DD^{\nu}_t$. Clearly, $\DD \cong \bigoplus_{\nu = 1}^{w+b} \DD^{\nu}$.
	For $1 \le \nu \le w$, let $\delta_{\nu} \in \mc D_{b_{\nu}}$ be an arbitrary lift of
	$\varphi_{b_{\nu}}(i_{b_{\nu}, \nu}(\sigma)) \in \mc D_{b_{\nu}} \otimes k$.
	Clearly, since $\mc D_{b_{\nu}}$ is a $W_{b_{\nu}}[F/F_{b_{\nu}}]$-module,
	we can define $\varphi_{b_{\nu}}'|_{\DD^{\nu}_{b_{\nu}}}$ by sending $i_{b_{\nu}, \nu}(\sigma) \mapsto \delta_{\nu}$.\\
	Assume now that $w < \nu \le w + b$. Let $\delta_{\nu}' \in \mc D_m$ be an arbitrary lift of
	$\varphi_m(i_{m, \nu}(\sigma)) \in \mc D_m \otimes k$.
	Then in $\mc D_n \otimes k$:
	\begin{align*}
		(\alpha_{\mc D, m, n} \otimes k)(\delta_{\nu}') &= 
		(\alpha_{\mc D, m, n} \otimes k)(\varphi_m(i_{m, \nu}(\sigma)))\\
		&= \varphi_n((\alpha_{\DD, m, n} \otimes k)(i_{m, \nu}(\sigma)))
		= 0.
	\end{align*}
	Therefore we have $\alpha_{\mc D, m, n}(\delta_{\nu}') = p \cdot z_{\nu}$ for 
	certain $z_{\nu} \in \mc D_n$. Since $\alpha_{\DD, t}$ is surjective for all $t \ge m$ by Corollary~\ref{cor:betas_injective} and $\DD \otimes k \cong \mc D \otimes k$, $\alpha_{\mc D, t}$ must be surjective for all $t \ge m$ by Nakayama's lemma. Thus there exists
	$u_{\nu} \in \mc D_m$ such that $\alpha_{\mc D, m, n}(u_{\nu}) = z_{\nu}$.
	Let $\delta_{\nu} := \delta_{\nu}' - p \cdot u_{\nu}$. Then
	$\alpha_{\mc D, m, n}(\delta_{\nu}) = 0$ and hence $N_{F/F_m}(\delta_{\nu}) = \beta_{\mc D, m, n}(\alpha_{\mc D, m, n}(\delta_{\nu})) = 0$. Therefore we can again define $\varphi_{b_{\nu}}'|_{\DD^{\nu}_m}$ by sending $i_{b_{\nu}, \nu}(\sigma) \mapsto \delta_{\nu}$.\\
	Finally, we put for all $1 \le \nu \le w+b$ and $1 \le t \le n$:
	\[
	\varphi_t'(i_{t, \nu}(\sigma)) :=
	\begin{cases}
		\beta_{\mc D, t, b_{\nu}}(\delta_{\nu}), & t \le b_{\nu},\\
		\alpha_{\mc D, b_{\nu}, t}(\delta_{\nu}), & t > b_{\nu}.
	\end{cases}
	\]
	One easily checks that this defines an $F$-equivariant map $\varphi_t'|_{\DD^{\nu}_t}$ for every $1 \le t \le n$ and $1 \le \nu \le w+b$.
	Finally, for every $1 \le \nu \le w+b$ and $1 \le t < b_{\nu}$:
	\begin{align*}
		\alpha_{\mc D, t}(\varphi'_t(i_{t, \nu}(\sigma))) &= 
		\alpha_{\mc D, t}(\beta_{\mc D, t, b_{\nu}}(\delta_{\nu}))
		= p \cdot \beta_{\mc D, t+1, b_{\nu}}(\delta_{\nu})\\
		&= \varphi'_{t+1}(p \cdot i_{t+1, \nu}(\sigma))
		= \varphi'_{t+1}(\alpha_{\DD, t}(i_{t, \nu}(\sigma))),\\
		\beta_{\mc D, t}(\varphi'_{t+1}(i_{t+1, \nu}(\sigma))) &= 
		\beta_{\mc D, t}(\beta_{\mc D, t+1, b_{\nu}}(\delta_{\nu}))\\
		&= \beta_{\mc D, t, b_{\nu}}(\delta_{\nu})
		= \varphi'_t(\beta_{\DD, t}(i_{t+1, \nu}(\sigma))).
	\end{align*}
	The proof that the constructed map commutes with $\alpha_t$ and $\beta_t$
	for $t \ge b_{\nu}$ is analogous.
\end{proof}
\begin{Corollary} \label{cor:G-YD_are_iso}
	Let $\DD$ be as in Proposition~\ref{prop:lifting_F_automorphisms}. Let $G$ be a finite group with a normal $p$-Sylow subgroup $F$.	Assume that $\DD'$, $\DD''$ are two $G$-Yakovlev diagrams that are isomorphic to $\DD$ as $F$-Yakovlev diagrams and such that $\DD' \otimes k \cong \DD'' \otimes k$ as $G$-Yakovlev diagrams. Then $\DD' \cong \DD''$ as $G$-Yakovlev diagrams.
\end{Corollary}
\begin{proof}
	Let $\varphi : \DD' \otimes k \to \DD'' \otimes k$ be an isomorphism of $G$-Yakovlev diagrams. By Proposition~\ref{prop:lifting_F_automorphisms} we can lift it to an $F$-isomorphism $\varphi' : \DD' \to \DD''$. Define $\varphi'' : \DD' \to \DD''$ by the following formula:
	\[
	\varphi''_t(x) := \frac 1M \sum_{i = 1}^M x_i \cdot \varphi'_t(x_i^{-1} \cdot x),
	\]
	where $G/F = \{ x_1 F, \ldots, x_M F \}$ (note that $p \nmid M$). Clearly, $\varphi''$ is $G$-equivariant.
	Moreover, since the maps $\alpha_t$, $\beta_t$ are $G$-equivariant,
	the maps $\varphi''_t$ commute with them. Finally:
	\begin{align*}
		(\varphi''_t \otimes k)(x) &= \frac 1M \sum_{i = 1}^M x_i \cdot (\varphi'_t \otimes k)(x_i^{-1} \cdot x)\\
		&= \frac 1M \sum_{i = 1}^M x_i \cdot \varphi_t(x_i^{-1} \cdot x) = \varphi_t(x).
	\end{align*}
	Therefore, $\varphi''_t$ is a surjection by Nakayama's lemma.
	This easily implies that $\len_W(\ker \varphi''_t) = \len_W \DD'_t - \len_W(\im \varphi''_t) = 0$. Hence $\varphi''_t$ is an isomorphism. This completes the proof. 
\end{proof}

\section{Proofs of Theorems~\ref{thm:crys_cyclic} and \ref{thm:cyclic_sylow} -- preparation} \label{sec:pf_of_cyclic}
We prove Theorems~\ref{thm:crys_cyclic} and \ref{thm:cyclic_sylow} in this and the next section. This section is devoted to some auxiliary results and in Section~\ref{sec:pf_of_cyclic2} we perform the inductive step.
We start by fixing some notation. Fix $n$ and assume that $\pi : X \to Y$ is an $F$-cover, where $F := \ZZ/p^n$. We abbreviate $m := m_{X/Y}$, $B := B_{X/Y}$, $e_Q := e_{X/Y, Q}$ for any $Q \in Y(\ol k)$, etc. Define:
\begin{itemize}
\item $\mc M := H^1_{\cris}(X)$ (considered as a $W[F]$-module)
\item[] and $\mc D = (\mc D_t, \alpha_t, \beta_t) := \Delta(\mc M)$,
	
\item the $W[F]$-module:
\begin{align*}
	\MM &:= W[F]^{2 (g_Y - 1)} \oplus \SS_m^2 \oplus \bigoplus_{\substack{Q \in B\\ Q \neq Q_0}} \JJ_{m_Q}^2
	\oplus \bigoplus_{Q \in B} \bigoplus_{t = 0}^{m_Q  - 1} \JJ_{m_Q - t}^{u_Q^{(t+1)} - u_Q^{(t)}}.
\end{align*}
and the Yakovlev diagram $\DD = (\DD_t, \alpha_{\DD, t}, \beta_{\DD, t}) := \Delta(\MM)$.
\end{itemize}
Our goal is to prove that $\mc M \cong \MM$.
Without loss of generality, we may assume that $B \subset Y(k)$. Indeed,
suppose that $B \subset Y(k')$, where $k'/k$ is a finite extension. Let $W' := W(k')$.
Then for any two $W[F]$-modules $M_1$, $M_2$ that are finite and free as $W$-modules one has
$M_1 \cong M_2$ as $W[F]$-modules if and only if $M_1 \otimes_W W' \cong M_2 \otimes_W W'$
as $W'[F]$-modules by the Reiner--Zassenhaus theorem (cf. \cite[Theorem~(30.25)]{Curtis-Reiner}).
Hence, if we prove that $\mc M \otimes_W W' \cong \MM \otimes_W W'$ as $W'[F]$-modules, then $\mc M \cong \MM$. Therefore, we can replace $k$ by $k'$.
\subsection{$\ZZ/p$-covers} \label{subsec:Z/p-covers}
Assume first that $\pi$ is a $\ZZ/p$-cover. By \cite[Theorem~4.1]{Garnek_Kontogeorgis_cyclic_de_Rham}
we have
\[
	H^1_{dR}(X) \cong 
	\begin{cases}
		k[F]^{2 g_Y} \oplus J_{p-1}^{\alpha}, & \textrm{ if $\pi$ is non-\'{e}tale,}\\
		k[F]^{2 (g_Y - 1)} \oplus k^{\oplus 2}, & \textrm{ if $\pi$ is \'{e}tale,}
	\end{cases}	
\]
where $\alpha = \sum_{Q \in B}(u_Q^{(1)} + 1) - 2$. Therefore by Corollary~\ref{cor:Zp_modules} we must have:
\[
H^1_{\cris}(X) \cong
\begin{cases}
	W[F]^{2 g_Y} \oplus \JJ_1^{\alpha}, & \textrm{ if $\pi$ is non-\'{e}tale,}\\
	W[F]^{2 (g_Y - 1)} \oplus W^{\oplus 2}, & \textrm{ if $\pi$ is \'{e}tale,}
\end{cases}	
\]
which completes the proof of Theorem~\ref{thm:crys_cyclic} in this case.
\subsection{$F$-invariants}
Let $\phi : X_1 \to X_2$ be a non-constant separable morphism of smooth projective curves over $k$. Recall that in this case there exist maps $\phi^* : H^1_{\cris}(X_2) \to H^1_{\cris}(X_1)$ and $\tr_{X_1/X_2} : H^1_{\cris}(X_1) \to H^1_{\cris}(X_2)$ (often denoted also $\phi_*$) such that
$\tr_{X_1/X_2} \circ \phi^* = (\deg \phi) \cdot \id$. In particular, since $H^1_{\cris}(X_2)$ is
torsion-free, $\phi^*$ is injective. 
Suppose now that $\phi$ is an $H$-cover. In this case $N_H : H^1_{\cris}(X_1) \to H^1_{\cris}(X_1)$ factors through $\tr_{X_1/X_2}$. In other words we have the following commutative diagram:
\begin{center}
	\begin{tikzcd}
		H^1_{\cris}(X_1) \arrow[rr, "N_H"] \arrow[rd, "\tr_{X_1/X_2}"] &                                               & H^1_{\cris}(X_1). \\
		& H^1_{\cris}(X_2) \arrow[ru, "\phi^*", hook] &              
	\end{tikzcd}
\end{center}
Clearly, the image of $H^1_{\cris}(X_2)$ is contained in $H^1_{\cris}(X_1)^H$.
\begin{Lemma} \label{lem:H1Y_=tr}
	Let $\pi : X \to Y$ be as in Theorem~\ref{thm:crys_cyclic}. Assume also that $m_{X/Y} = n$. Then
	$\tr_{X/Y} : H^1_{\cris}(X) \to H^1_{\cris}(Y)$ is onto.
\end{Lemma}
\begin{proof}
	By \cite[Lemma~4.5]{Garnek_Kontogeorgis_cyclic_de_Rham} the trace map $\tr_{X/Y} : H^1_{dR}(X) \to H^1_{dR}(Y)$
	is onto. Thus by Nakayama's lemma, $\tr_{X/Y} : H^1_{\cris}(X) \to H^1_{\cris}(Y)$ is an epimorphism as well.
\end{proof}
\begin{Lemma} \label{lem:invariants}
	Let $\pi : X \to Y$ be as in Theorem~\ref{thm:crys_cyclic}. Then $\pi^*$ induces an isomorphism $H^1_{\cris}(Y) \cong H^1_{\cris}(X)^F$.
\end{Lemma}
\begin{proof}
	By induction it suffices to prove the case $F = \ZZ/p$.
	Assume first that $\pi$ is an \'{e}tale cover. Then:
	\[
		H^1_{dR}(Y) \cong H^1_{\cris}(Y) \otimes_W k 
		\subset H^1_{\cris}(X)^F \otimes_W k \subset H^1_{dR}(X)^F.
	\]
	On the other hand, $H^1_{dR}(X)^F = H^1_{dR}(Y)$ by \cite[Lemma~4.3]{Garnek_Kontogeorgis_cyclic_de_Rham}.
	Hence by Nakayama's lemma $H^1_{\cris}(Y) \cong H^1_{\cris}(X)^F$.
	If $\pi$ isn't \'{e}tale then $H^1_{\cris}(X) \cong W[F]^{2g_Y} \oplus \JJ_1^{\alpha}$ (cf. Subsection~\ref{subsec:Z/p-covers}).
	Thus $H^2(F, \mc M) = 0$ by Proposition~\ref{prop:Delta_J_otimes_k}, i.e. $H^1_{\cris}(X)^F = N_F \cdot H^1_{\cris}(X)$. By Lemma~\ref{lem:H1Y_=tr} we have $N_F \cdot H^1_{\cris}(X) = H^1_{\cris}(Y)$. This finishes the proof.
\end{proof}
\begin{Proposition} \label{prop:H2=0|H1=0}
	\begin{enumerate}[(1)]
		\item[]
		\item If $m_{X/Y} = n$ then $H^2(F, \mc M) = 0$.
		
		\item If $\pi$ is \'{e}tale then $H^1(F, \mc M) = 0$.
	\end{enumerate}
\end{Proposition}
\begin{proof}
	(1) By Lemmas~\ref{lem:H1Y_=tr} and~\ref{lem:invariants}:
	\[
	H^1_{\cris}(Y) \cong H^1_{\cris}(X)^F \cong N_F \cdot H^1_{\cris}(X).
	\]
	Thus by~\eqref{eqn:H2_definition} we have $H^2(F, \mc M) = \mc M^F/(N_F \cdot \mc M) = 0$.\\
	(2) We prove this by induction on $|F|$. For $|F| = p$ this is clear by Subsection~\ref{subsec:Z/p-covers}.
	Assume now that this holds for cyclic $p$-groups of order less than $|F|$.
	Then by Lemma~\ref{lem:invariants} $\mc M^{F_1} = H^1_{\cris}(X_1)$, where $X_1 := X/F_1$.
	Thus by the induction assumption for $X_1 \to Y$ we obtain $H^1(F/F_1, \mc M^{F_1}) = 0$.
	Also by the induction assumption for $X \to X_1$, $H^1(F_1, \mc M) = 0$. Therefore 
	$H^1(F, \mc M) = 0$ by the inflation--restriction exact sequence.
\end{proof}
\begin{Corollary} \label{cor:H2(M)}
	For any $t > m$ we have an isomorphism of $W[F]$-modules:
	\[
		H^2(F_t, \mc M)[p] \cong k^{\oplus 2}.
	\]
\end{Corollary}
\begin{proof}
	Write $X_{\et} := X/F_m$, $\mc M_{\et} := H^1_{\cris}(X_{\et})$, $\mc N := H^1_{dR}(X_{\et})$ and $\ol F_i := F_i/F_m$ for any $i \ge m$.
	Note that the $\ol F_n$-cover $X_{\et} \to Y$ is \'{e}tale, while the $F_m$-cover $X \to X_{\et}$ satisfies  the assumptions of Lemma~\ref{lem:H1Y_=tr}.
	By Lemma~\ref{lem:H1Y_=tr} and Lemma~\ref{lem:invariants} we have $\mc M^{F_m} \cong \mc M_{\et}$ and $N_{F_m} \mc M = \mc M_{\et}$. Hence:
	\[
		H^2(F_t, \mc M) = \frac{\mc M^{F_t}}{N_{F_t} \mc M} = \frac{(\mc M^{F_m})^{\ol F_t}}{N_{\ol F_t}(N_{F_m} \mc M)}
		= H^2(\ol F_t, \mc M_{\et}).
	\]
	By Lemma~\ref{lem:Tor_ses} and Proposition~\ref{prop:H2=0|H1=0}~(2) we see that $H^2(\ol F_t, \mc M_{\et})[p] \cong H^1(\ol F_t, \mc N)$ for any $t > m$. Observe that $\mc N \cong k[\ol F_n]^{2(g_Y - 1)} \oplus k^{\oplus 2}$ as a $k[\ol F_n]$-module by \cite[Theorem~4.1]{Garnek_Kontogeorgis_cyclic_de_Rham}.
	Hence:
	\[
		H^1(\ol F_t, \mc N) \cong H^1(\ol F_t, k[\ol F_n]^{2(g_Y - 1)} \oplus k^{\oplus 2}) \cong k^{\oplus 2}. \qedhere
	\]
\end{proof}
\section{Proofs of Theorems~\ref{thm:crys_cyclic} and \ref{thm:cyclic_sylow}} \label{sec:pf_of_cyclic2}
In this section we continue the proof of Theorem~\ref{thm:crys_cyclic}. Recall
that we have verified it for $n = 1$ (cf. Subsection~\ref{subsec:Z/p-covers}).
We proceed by induction on $n$ in the following steps:
\begin{enumerate}
	\item \bb{Step I: reduction.} 
	\item[] We show that $\mc D \otimes k \cong \DD \otimes k$.
	
	\item \bb{Step II: relation with the fixed points module.}
	\item[] We show that $\Fix_{F_1}(\mc D)_t = p \cdot \mc D_{t+1}$, $\Fix_{F_1}(\DD)_t = p \cdot \DD_{t+1}$ and that $\Fix_{F_1}(\mc D) \cong \Fix_{F_1}(\DD)$.
	
	\item \bb{Step III: Isomorphism of Yakovlev diagrams.}
	\item[] We show that $\mc D \cong \DD$.
	
	\item \bb{Step IV: isomorphism of modules.}
	\item[] We deduce from $\mc D \cong \DD$ that $\mc M \cong \MM$.
\end{enumerate}
\subsection{Reduction} \label{subsec:reduction}
The comparison between de Rham and crystalline cohomology implies that
$\mc M \otimes_W k \cong H^1_{dR}(X)$. Hence, by \cite[Theorem 4.1]{Garnek_Kontogeorgis_cyclic_de_Rham} and Lemma~\ref{lem:JJ_otimes_k} we have:
\begin{align*}
	\mc M \otimes_W k &\cong J_{p^n}^{2 (g_Y - 1)} \oplus J_{p^n - p^{n-m} + 1}^2 \oplus \bigoplus_{\substack{Q \in B\\ Q \neq Q_0}} J_{p^n - p^n/e_{Q}}^2
	\oplus \bigoplus_{Q \in B} \bigoplus_{t = 0}^{m_Q  - 1} J_{p^n - p^{n+t}/e_Q}^{u_Q^{(t+1)} - u_Q^{(t)}}\\
	& \cong \MM \otimes_W k.
\end{align*}
\begin{Corollary} \label{cor:mcD_otimes_k_is_E}
	The $F$-Yakovlev diagrams $\mc D \otimes k$ and $\Delta(\mc M \otimes k)^{(m)}$ are isomorphic.
\end{Corollary}
\begin{proof}
	By Proposition~\ref{prop:H2=0|H1=0}~(1), we have $H^2(F_t, \mc M) = 0$ for $t \le m$, which implies $\mc D_t \otimes k \cong \Delta(\mc M \otimes k)_t$ by Lemma~\ref{lem:Tor_ses}. Assume now that $t > m$. Using again Lemma~\ref{lem:Tor_ses} and the equality
	$H^2(F_m, \mc M) = 0$ we obtain the following commutative diagram with exact rows:
	\begin{equation} \label{eqn:diagram_D_otimes_k_vs_Delta}
		\begin{tikzcd}
			0 \arrow[r] & \mc D_m \otimes k \arrow[d] \arrow[r] & \Delta(\mc M \otimes k)_m \arrow[d] \arrow[r] & 0 \arrow[d] \arrow[r]          & 0 \\
			0 \arrow[r] & \mc D_t \otimes k \arrow[r]           & \Delta(\mc M \otimes k)_t \arrow[r]           & {H^2(F_t, \mc M)[p]} \arrow[r] & 0
		\end{tikzcd}
	\end{equation}
	Moreover, by Proposition~\ref{prop:Delta_J_otimes_k} we have
	\begin{align*}
		\coker(\Delta(\mc M \otimes k)_m \to \Delta(\mc M \otimes k)_t)
		&\cong \coker(\Delta(\MM \otimes k)_m \to \Delta(\MM \otimes k)_t)\\
		&\cong \coker(\Delta(\SS_m^2 \otimes k)_m \to \Delta(\SS_m^2 \otimes k)_t)\\
		&\cong k^2.
	\end{align*}
	Hence, applying the Snake Lemma to the diagram above, we obtain the short exact sequence:
	\begin{align*}
		0 &\to \coker(\mc D_m \otimes k \to \mc D_t \otimes k) \to \coker(\Delta(\mc M \otimes k)_m \to \Delta(\mc M \otimes k)_t)\\
		&\to H^2(F_t, \mc M)[p] \to 0.
	\end{align*}
	By Corollary~\ref{cor:H2(M)} the last two terms in this exact sequence are isomorphic,
	which means that the map $\mc D_m \otimes k \to \mc D_t \otimes k$ is surjective.
	This allows to conclude from the diagram~\ref{eqn:diagram_D_otimes_k_vs_Delta} that $\mc D_t \otimes k \cong \im(\Delta(\mc M \otimes k)_m \to \Delta(\mc M \otimes k)_t)$. 
\end{proof}

\begin{Corollary} \label{cor:same_reductions}
	We have an isomorphism $\mc D \otimes_W k \cong \DD \otimes_W k$ of $F$-Yakovlev diagrams over $W$.
\end{Corollary}
\begin{proof}
	By Proposition~\ref{prop:Delta_J_otimes_k} one obtains $\Delta(\MM) \otimes k \cong \Delta(\MM \otimes k)^{(m)}$.
	Hence by Corollary~\ref{cor:mcD_otimes_k_is_E}:
	\[
		\Delta(\MM) \otimes k \cong \Delta(\MM \otimes k)^{(m)}
		\cong \Delta(\mc M \otimes k)^{(m)} \cong \mc D \otimes k. \qedhere
	\]
\end{proof}
\begin{Corollary} \label{cor:ker_Deltan_Delta_1=p_Delta_n}
	For every $t \ge 1$ we have:
	\[
		\ker(\beta_{1, t+1} : \mc D_{t+1} \to \mc D_1) = p \cdot \mc D_{t+1}.
	\]
\end{Corollary}
\begin{proof}
Clearly, $p \cdot \mc D_{t+1}$ is contained in $\ker(\mc D_{t+1} \to \mc D_1)$ by~\eqref{eqn:pT_Delta_subset_Fix_Delta}. To end the proof it suffices to show that for every $t$, the map
\[
\beta_t : \mc D_{t+1}/p \mc D_{t+1} \to \mc D_t/p \mc D_t
\]
is injective. Since $\mc D \otimes k \cong \DD \otimes k$ by Corollary~\ref{cor:same_reductions}, the statement
follows from Corollary~\ref{cor:betas_injective}.
\end{proof}

\subsection{Relation with the fixed points module} \label{subsec:fixed}
Lemma~\ref{lem:invariants} implies that $\mc M^{F_1} \cong H^1_{\cris}(X')$.
Abbreviate $m_Q' := m_{X'/Y, Q}$.
By induction assumption:
\begin{align*}
	H^1_{\cris}(X') \cong W[F']^{2 (g_Y - 1)} \oplus \SS_{m_{X'/Y}}^2 \oplus \bigoplus_{\substack{Q \in B\\ Q \neq Q_0}} \JJ_{m_Q'}^2
	\oplus \bigoplus_{Q \in B} \bigoplus_{t = 0}^{m_{X'/Y, Q}  - 1} \JJ_{m_{X'/Y, Q} - t}^{u_{X'/Y, Q}^{(t+1)} - u_{X'/Y, Q}^{(t)}}.
\end{align*}
Note that $m_{X'/Y} = \max(m_{X/Y} - 1, 0)$ and $u_{X'/Y, Q}^{(t)} = u_{X/Y, Q}^{(t)}$ for any $t < m_{X/Y, Q}$ and any $Q \in B$ (since the upper ramification jumps are compatible with passage to quotients, cf. \cite[Proposition IV.\S 3.14]{Serre1979}).
Using~\eqref{eqn:fix_Ji_Si} one easily checks that $H^1_{\cris}(X') \cong \MM^{F_1}$. To summarize, we obtain $\mc M^{F_1} \cong \MM^{F_1}$
and $\Fix_{F_1}(\mc D) \cong \Fix_{F_1}(\DD)$ by Lemma~\ref{lem:operations_on_yd}. On the other hand, by Corollary~\ref{cor:ker_Deltan_Delta_1=p_Delta_n} and Corollary~\ref{cor:betas_injective} we have
$\Fix_{F_1}(\mc D)_t = p \cdot \mc D_{t+1}$ and $\Fix_{F_1}(\DD)_t = p \cdot \DD_{t+1}$. 
\subsection{Isomorphism of Yakovlev diagrams} \label{subsec:iso_of_YD}
By Corollary~\ref{cor:same_reductions} we know that $\DD_t \otimes k \cong \mc D_t \otimes k$ for all $1 \le t \le n$ and $p \DD_t \cong p \mc D_t$ for all $1 \le t \le n$
(for $t = 1$ both sides are zero modules). Hence, by Lemma~\ref{lem:pM_and_M/p_determines_M} we have $\DD_t \cong \mc D_t$
as $W$-modules for all $1 \le t \le n$. On the other hand, by Corollary~\ref{cor:same_reductions} there exists an isomorphism $\varphi : \DD \otimes k \to \mc D \otimes k$
of $F$-Yakovlev diagrams. Using Proposition~\ref{prop:lifting_F_automorphisms} we can lift it to a morphism $\varphi' : \DD \to \mc D$. It is surjective by Nakayama's lemma, and hence
$\len_W(\ker \varphi_t') = \len_W(\DD_t) - \len_W(\mc D_t) = 0$. Thus $\varphi'$ is an isomorphism.
\subsection{Isomorphism of modules}
By the previous steps we obtain $\mc D \cong \DD$. Therefore by
Theorem~\ref{thm:yako1} we have:
\[
	\mc M \oplus \bigoplus_{i = 0}^n W[F/F_i]^{a_i}
	\cong \MM \oplus \bigoplus_{i = 0}^n W[F/F_i]^{b_i}
\]
for some $a_i, b_i \ge 0$. By tensoring this isomorphism with $k$ and noting that
$\mc M \otimes_W k \cong \MM \otimes_W k$, we see that
\[
	\bigoplus_{i = 0}^n k[F/F_i]^{a_i}
	\cong \bigoplus_{i = 0}^n k[F/F_i]^{b_i}.
\]
This implies that $a_i = b_i$ for $i = 0, \ldots, n$, since $k[F/F_i]$ are pairwise non-isomorphic indecomposable $k[F]$-modules. Hence, since $\Mod_{W[F]}^{\mathrm{ff}}$ is a Krull--Schmidt category,
we see that $\mc M \cong \MM$. This completes the proof of Theorem~\ref{thm:crys_cyclic}.
\subsection{Proof of Theorem~\ref{thm:cyclic_sylow}}
We prove now Theorem~\ref{thm:cyclic_sylow}.
Let $k$, $G$, $F$ and $X$ be as in Theorem~\ref{thm:cyclic_sylow}.
By Lemma~\ref{lem:conlon} we can assume without loss of generality that
$G$ is $p$-hypo-elementary. In particular $F$ is normal in $G$ and $N_1 = N_2 = \cdots = N_n = G$.
Let $\mc M := H^1_{\cris}(X)$ and $\mc D := \Delta(\mc M)$.
The isomorphism in Corollary~\ref{cor:mcD_otimes_k_is_E} can be promoted to an isomorphism of $G$-Yakovlev diagrams. Thus, since the $G$-structure of $H^1_{dR}(X)$ is determined by the ramification data and the genus of $X$, the $G$-structure of $\mc D \otimes k$
is determined by the ramification data and the genus of $X$ as well. 
By Corollary~\ref{cor:G-YD_are_iso}, the $G$-structure of $\mc D$ is determined by the $G$-structure of $\mc D \otimes k$, and hence depends only on the ramification data and on the genus of~$X$.
Thus by Theorem~\ref{thm:yako2} the ramification
data determines $H^1_{\cris}(X)$ up to a trivial source module. Note however that by \cite[Corollary~2.6.3]{Benson_modular_representation_theory}
a trivial source $W[G]$-module $M$ is uniquely determined by its reduction $M \otimes_W k$.
We conclude by noting that $H^1_{\cris}(X) \otimes_W k \cong H^1_{dR}(X)$ is determined
by the ramification data by \cite[Main Theorem]{Garnek_Kontogeorgis_cyclic_de_Rham}.
\bibliography{bibliografia}
%

\end{document}